\documentclass[12pt]{amsart}

\usepackage{amsthm,amssymb,verbatim,amsmath, amsfonts, amscd}
\usepackage{xcolor}

\usepackage{amssymb}
\usepackage{amsthm}

\usepackage{amsthm,amssymb,verbatim,amsmath, amsfonts, amscd 
,a4wide,color}
\usepackage{ascmac}    

\usepackage{hyperref}              
\usepackage{verbatim} 

\theoremstyle{plain}\newtheorem{Theorem}{Theorem}[section]
\theoremstyle{plain}
\theoremstyle{plain}
\theoremstyle{plain}\newtheorem{Lemma}[Theorem]{Lemma}
\theoremstyle{plain}
\theoremstyle{definition}\newtheorem{Definition}[Theorem]{Definition}
\theoremstyle{definition}
\theoremstyle{definition}
\theoremstyle{definition}
\theoremstyle{definition}
\theoremstyle{definition}\newtheorem{Notation}[Theorem]{Notation}
\theoremstyle{definition}
\theoremstyle{definition}
\theoremstyle{definition}
\theoremstyle{definition}
\theoremstyle{definition}
\theoremstyle{definition}
\theoremstyle{definition}\newtheorem{Notation/Definition}
[Theorem]{Notation/Definition}
\theoremstyle{definition}

\def\CF{{\mathcal{F}}}

\def\Aut{\mathrm{Aut}}

\def\ker{\mathrm{ker}}

\def\Ind{\mathrm{Ind}}           
\def\Inn{\mathrm{Inn}}

\def\Res{\mathrm{Res}}

\newcommand{\Sc}{{\text{Sc}}}

\newcommand{\Out}{{\text{Out}}}

\begin{document}

\title{The Brauer indecomposability of Scott modules with wreathed $2$-group vertices}
\date{\today}
\author{{Shigeo Koshitani and {\.I}pek Tuvay}}
\dedicatory{Dedicated to Professor Hiroyuki Tachikawa on his 90th birthday.}
\address{Center for Frontier Science,
Chiba University, 1-33 Yayoi-cho, Inage-ku, Chiba 263-8522, Japan.}
\email{koshitan@math.s.chiba-u.ac.jp}
\address{Mimar Sinan Fine Arts University, Department of Mathematics, 34380, Bomonti, \c{S}i\c{s}li, Istanbul, Turkey}
\email{ipek.tuvay@msgsu.edu.tr}

\thanks{The first author was partially supported by
the Japan Society for Promotion of Science (JSPS), Grant-in-Aid for Scientific Research
(C)19K03416, 2019--2021. The second author was partially supported by Mimar Sinan Fine Arts University
Scientific Research Unit with project number 2019-28}


\begin{abstract}
{
 
We give a sufficient condition for the
$kG$-Scott module with vertex $P$ to remain
indecomposable under taking the Brauer construction for any subgroup $Q$ of $P$
as $k[Q\,C_G(Q)]$-module,
where $k$ is a field of characteristic $2$, 
and $P$ is a wreathed $2$-subgroup of a finite group $G$.
This generalizes results for the cases where $P$ is abelian and some others.
The motivation of this paper is that the Brauer indecomposability of
a $p$-permutation bimodule ($p$ is a prime) is one of the key steps in order
to obtain a splendid stable equivalence of Morita type 
by making use of the gluing method
that then can possibly lift to a splendid derived equivalence.
}
\end{abstract}

\maketitle

\section{Introduction and notation}

\noindent
In modular representation theory of finite groups,
the Brauer construction $M(P)$ of a $p$-permutation module $M$ with respect to 
a $p$-subgroup $P$ of a finite group $G$ plays a very important role,
where $k$ is an algebraically closed field of characteristic $p>0$.
It canonically becomes a $p$-permutation module over $kN_G(P)$ (see p.402 in \cite{Bro}).
In their paper \cite{KKM} Kessar, Kunugi and Mitsuhashi introduce a notion {\sl Brauer indecomposability}.
Namely, $M$ is called {\sl Brauer indecomposable} if the restriction
module ${\mathrm{Res}}\,^{N_G(Q)}_{Q\,C_G(Q)} (M(Q))$ is indecomposable or zero
as $k(Q\,C_G(Q))$-module for any subgroup $Q$ of $P$.
Actually in order to get a kind of equivalence between
the module categories for the principal blocks $A$ and $B$ 
of the group algebras $kG$ and $kH$ respectively
(where $H$ is another finite group), e.g. 
in order to prove Brou\'e's abelian defect group conjecture, 
we usually first of all  would
have to face a situation such that $A$ and $B$
are stably equivalent of Morita type.
In order to get the stable equivalence, 
we often want to check whether the $k(G\times H)$-Scott module
${\mathrm{Sc}}(G\times H, \Delta P)$ with a vertex  
$\Delta P:=\{(u,u)\in P\times P \}$ induces a stable
equivalence of Morita type between $A$ and $B$, where
$P$ is a common Sylow $p$-subgroup of $G$ and $H$
by making use of 
the gluing method due to Brou\'e (see \cite{Bro94}), and also
Rickard, Linckelmann and Rouquier.
If this is the case, then ${\mathrm{Sc}}(G\times H, \Delta P)$
have to be Brauer indecomposable.
Therefore it should be very important to know whether
${\mathrm{Sc}}(G\times H, \Delta P)$
is Brauer indecomposable or not.
This is the motivation why we have written this paper.
Actually, our main result is the following:
\begin{Theorem}\label{product}
Let $G$ and $G'$ be finite groups with a common Sylow $2$-subgroup 
$P\cong C_{2^n}\wr C_2$ for some $n\geq 2$, and assume that the two 
fusion systems of $G$ and $G'$ over $P$ are the same,
namely ${\mathcal F}_P(G)={\mathcal F}_P(G')$.
Then the Scott module
${\mathrm{Sc}}(G \times G', \Delta P)$ is Brauer indecomposable.
\end{Theorem}

This theorem in a sense generalizes \cite{KKM, KKL, KL, KL2}, and 
there are results on Brauer indecomposability of Scott modules also
in \cite{ KT19a, KT19b,  T}.

\begin{Notation} Besides the notation explained above we need the following notation and terminology.
In this paper $G$ is always a finite group, and $k$ is an algebraically closed field of characteristic $p>0$.

By $H\leq G$ and $H \unlhd G$ we mean that $H$ is a subgroup and a normal subgroup of $G$, 
respectively.
We write $O_{p'}(G)$ for the largest normal ${p'}$-subgroup of $G$,
and $Z(G)$ for the center of $G$.
For $x, y\in G$ we set $x^y:=y^{-1}xy$ and $^y\!x:= yxy^{-1}$.
Further for $H\leq G$ and $g\in G$ 
we set $^g\!H:=\{\, ^g\!h \, |\, \forall h\in H \}$.
For two groups $K$ and $L$ we write $K\rtimes L$ for a semi-direct product of $K$ by $L$
where $K \unlhd (K\rtimes L)$.
For two subgroups $K, L$ of $G$ we denote by $G=K*L$ the central product
of $K$ and $L$ with respect to $G$. 
For a positive integer $m$, we mean by $C_m$ and $S_m$
the cyclic group of order $m$ and the symmetric group of degree $m$, respectively.
Furthermore  $D_{2^m}$ and $Q_{2^m}$ for $m\geq 3$
are  the dihedral and the generalized quaternion groups of order $2^m$, respectively.
Our main character actually in this paper is {\it the wreathed product $2$-group}
$ C_{2^n}\wr C_2 \cong (C_{2^n}\times C_{2^n})\rtimes C_2  \text{ for }n\geq 2.$

We mean by a $kG$-module a finitely generated left $kG$-module unless stated otherwise.
For a $kG$-module and a $p$-subgroup $P$ of $G$ the Brauer construction $M(P)$ is
defined as in \S27 of \cite{Th} or p.402 in \cite{Bro}. For such $G$ and $P$ we write $\mathcal F_P(G)$ for
the fusion system of $G$ over $P$ as in I.Definition 1.1 of \cite{AKO}.
For two $kG$-modules $M$ and $L$, we write $L\,|\,M$ if $L$ is (isomorphic to) a direct summand of $M$
as a $kG$-module.
We write $k_G$ for the trivial $kG$-module.
For $H\leq G$, a $kG$-module $M$ and a $kH$-module $N$, 
we write ${\mathrm{Res}}^G_H(M)$ for the restriction module of
$M$ from $G$ to $H$, and ${\mathrm{Ind}}_H^G(N)$ for the induced module of $N$ from $H$ to $G$.
For $H\leq G$ we denote the (Alperin-)Scott $kG$-module with respect to $H$ 
by ${\mathrm{Sc}}(G,H)$.
By definition, ${\mathrm{Sc}}(G,H)$ is the unique indecomposable direct summand
of ${\mathrm{Ind}}_H^G(k_H)$ which contains $k_G$ in its top.
We refer the reader to 
\S2 of \cite{Bro} and  in \S 4.8.4 of \cite{NT} for further details on Scott modules.

For the other notations and terminologies, see the books
\cite{NT}, \cite{Gor} and \cite{AKO}.
\end{Notation}

The organization of this paper is as follows.
In \S2 we shall give several theorems and general lemmas
that are useful of our aim, and in \S3 we shall give several lemmas which are on
finite groups with the wreathed $2$-subgroup. 
We shall prove in \S4 several lemmas and a theorem
which are on the Scott module whose vertex is the wreathed $2$-group.
Finally in \S5 our main theorem shall be proved.

\section{General lemmas}
\noindent
In this section we give several lemmas and theorems which play very important role for our purpose.
\begin{Lemma}[{\cite[Lemma 4.2]{KL}}]\label{BaerSuzuki}
Let $Q$ be a normal $2$-subgroup of $G$ such that $G/Q \cong S_3$.
Assume further that there is an involution $t \in G  -  Q $. Then $G$ has a subgroup $H$ such that
$t \in H \cong S_3$.
\end{Lemma}

\begin{Theorem}[{Theorem 1.3 of \cite{IK}}]\label{IK1}
Assume that $P$ is a $p$-subgroup of $G$ and
${\mathcal F}_P(G)$ is saturated. Then the following assertions are equivalent:
\begin{enumerate}
\item[\rm (a)] ${\mathrm{Sc}}(G,P)$ is Brauer indecomposable.
\item[\rm (b)] ${\mathrm{Res}}\,_{Q\,C_G(Q)}^{N_G(Q)} \Big({\mathrm{Sc}}(N_G(Q), N_P(Q))\Big)$ 
is indecomposable for each fully ${\mathcal F}_P(G)$-normalized
subgroup $Q$ of $P$.
\end{enumerate}
If these conditions are satisfied, then $({\mathrm{Sc}}(G,P))(Q)
\cong {\mathrm{Sc}}(N_G(Q),N_P(Q))$ for each fully ${\mathcal F}_P(G)$-normalized
subgroup $Q \leq P$.
\end{Theorem}

Actually the following result due to Ishioka and Kunugi is so useful for our aim.

\begin{Theorem}[{Theorem 1.4 of \cite{IK}}]\label{IK2}
Assume that $P$ is a $p$-subgroup of $G$ and ${\mathcal F}_P(G)$ is
saturated. Let $Q$ be a fully ${\mathcal F}_P(G)$-normalized subgroup of $P$. Assume further that there exists a subgroup
$H_Q$ of $N_G(Q)$ satisfying the following conditions:
\begin{enumerate}
\item[\rm (a)] $N_P(Q)$ is a Sylow $p$-subgroup of  $H_Q$ and
\item[\rm (b)] $|N_G(Q):H_Q|=p^a$ for an integer $a\geq 0$.
\end{enumerate}
Then ${\mathrm{Res}}\,_{Q\,C_G(Q)}^{N_G(Q)} \Big({\mathrm{Sc}}(N_G(Q), N_P(Q))\Big)$ is indecomposable.
\end{Theorem}

\begin{Lemma}[{\cite{K}}]\label{centralizer}
Let $G$ be a finite group with a $p$-subgroup $P$. 
\begin{enumerate}
\item[\rm (i)]
For every subgroup $Q \leq P$, $Q\,C_P(Q)$ is a maximal element of 
the set 
$N_P(Q) \cap_{N_G(Q)} Q\,C_G(Q):=\{ {}^g\!N_P(Q) \cap Q\,C_G(Q) \ | \ g \in N_G(Q)\}.$
\item[\rm (ii)] We have that
 $${\mathrm{Sc}}(Q\,C_G(Q), Q\,C_P(Q))\,\Big|\, 
\Res\,_{Q\,C_G(Q)}^{N_G(Q)}\Big({\mathrm{Sc}}(N_G(Q), N_P(Q))\Big),$$
and
 $${\mathrm{Sc}}(C_G(Q), C_P(Q))\,\Big|\,
\Res\,_{C_G(Q)}^{N_G(Q)}\Big({\mathrm{Sc}}(N_G(Q), N_P(Q))\Big).$$
\end{enumerate}
\end{Lemma}

\begin{proof}
Let $g \in N_G(Q)$, then we have that 
$${}^g\!(N_P(Q)) \cap Q\,C_G(Q) = N_{^g\!P}(Q) \cap Q\,C_G(Q)=Q\,C_{^g\!P}(Q).$$ 
Moreover, $|Q\,C_P(Q)|=|{}^g (Q\,C_P(Q))|=|Q\,C_{^g\!P}(Q)|$ since $g \in N_G(Q)$. So each 
element in $N_P(Q) \cap_{N_G(Q)} Q\,C_G(Q)$ has the same order. Hence, by letting $g=1$, 
$Q\,C_P(Q)=N_P(Q) \cap Q\,C_G(Q) $ 
is a maximal element. Therefore, from \cite[Theorem 1.7]{K},
$$\Sc(Q\,C_G(Q), Q\,C_P(Q))\,\Big|\, \Res\,_{Q\,C_G(Q)}^{N_G(Q)}\Big(\Sc(N_G(Q), N_P(Q))\Big).$$
The final assertion follows from a similar argument.
\end{proof}

\begin{Lemma}\label{centralizer2}
Let $G$ be a finite group with a $p$-subgroup $P$.
Assume that $\CF_P(G)$ is saturated, and fix a subgroup $Q\leq P$
which is fully $\CF_P(G)$-normalized.
If  
$
{\mathrm{Res}}\,^{N_G(Q)}_{Q\,C_G(Q)}\Big( {\mathrm{Sc}}(G,P)(Q)\Big)
$ is indecomposable, then we have that
$$ 
{\mathrm{Res}}\,^{N_G(Q)}_{Q\,C_G(Q)}\Big( {\mathrm{Sc}}(G,P)(Q)\Big)
\cong {\mathrm{Sc}}(Q\,C_G(Q), \, Q\,C_P(Q)).
$$
\end{Lemma}

\begin{proof}
Set $M:={\mathrm{Sc}}(G,P)$.
First, since from the assumption that 
$
{\mathrm{Res}}\,^{N_G(Q)}_{Q\,C_G(Q)}\Big( M(Q)\Big)
$ 
is indecomposable, we get also 
that $M(Q)$ is indecomposable as $kN_G(Q)$-module. 
Hence \cite[the first four lines of the proof of Theorem 1.3]{IK}
implies that 
$$
 M(Q) \cong {\mathrm{Sc}}(N_G(Q), N_P(Q)),
$$
so that 
\begin{equation}\label{M(Q)}
{\mathrm{Res}}\,^{N_G(Q)}_{Q\,C_G(Q)}\Big( {\mathrm{Sc}}(N_G(Q), N_P(Q))\Big)
\text{ is indecomposable}
\end{equation}
by the assumption.
On the other hand, Lemma \ref{centralizer}(ii) yields that
\begin{equation}\label{centralizerQ}
{\mathrm{Sc}}(Q\,C_G(Q), Q\,C_P(Q))
\,\Big|\,
\Res\,_{Q\,C_G(Q)}^{N_G(Q)}\Big({\mathrm{Sc}}(N_G(Q), N_P(Q))\Big).
\end{equation}
Therefore, (\ref{M(Q)}) and (\ref{centralizerQ}) imply the assertion.
\end{proof}

The following lemma is perhaps well-known for experts
(see \cite[p.102]{KKM} for instance. See also \cite[Proposition 1.5]{BP}).

\begin{Lemma}\label{Kunugi}
Suppose that $M$ is a (possibly decomposable) trivial source $kG$-module, and
let $P\leq G$ be a $p$-subgroup of $G$, and that $Q\unlhd P$.
Then, it holds that
$$(M(Q))(P)\cong\Res\,^{N_G(P)}_ {N_G(P)\cap N_G(Q)}\,(M(P)).$$
\end{Lemma}

\begin{Lemma}\label{Lin2.5}
Suppose that $P\leq G$ is a $p$-subgroup of $G$, and set $\CF:=\CF_P(G)$.
If $\CF$ is saturated and if $Q$ is a fully $\CF$-normalized subgroup of $P$, then
$$N_P(Q)/C_P(Q) \in{\mathrm{Syl}}_p(N_G(Q)/C_G(Q))$$
and that $Q$ is fully $\CF$-centralized in $P$.
\end{Lemma}
\begin{proof} Follows from \cite[Proposition 2.5]{L}.
\end{proof}

\section{{}Lemmas on a finite group $G$ having a $2$-subgroup $C_{2^n}\wr C_2$}

\noindent
We give in this section kind of group theory for our particular interest
which is the wreathed $2$-group $C_{2{^n}}\wr C_2$.

\begin{Definition}
A wreathed $2$-group $C_{2^n}\wr C_2$ is a group 
$P$ which is a wreath product of a cyclic group of order $2^n$ 
for $n \geq 2$ with a cyclic group
of order $2$. More precisely, $P$ is generated by three elements $a, b$ and $t$
and defined as follows,  where
$$ a^{2^n}= b^{2^n}=1, \ t^2=1, \ a^t=b, \ b^t=a, \ a b=b a,$$
$$P:=C_{2^n}\wr C_2 =(\langle a\rangle\times\langle b\rangle)\rtimes\langle t\rangle
\cong (C_{2^n}\times C_{2^n})\rtimes C_2.$$
Notice $P$ has order $2^{2n+1}$ and $Z(P)=\langle ab \rangle \cong C_{2^n}$. 
Define the base subgroup $P_0$ of $P$ by 
$$P_0:=\langle a\rangle\times\langle b\rangle\cong C_{2^n}\times C_{2^n}.$$
We remark here that all the notation $P$, $P_0$, $a, b, t$ and $n$ shall be fixed
from now on till the end of this paper.
\end{Definition}

\begin{Lemma}\label{CentralizerOutsideP0}
If $x\in P-P_0$,
then $|C_P(x)|=2^{n+1}$ and $C_P(x) \cap P_0=Z(P)(=\langle ab\rangle\cong C_{2^n}).$
\end{Lemma}

\begin{proof}
We can write $x:=a^ib^jt$ for some integers $i$ and $j$. Then $x^2=(ab)^{i+j} \in \langle ab \rangle=Z(P).$ Moreover,
since $x \not\in P_0$, we have that 
$\langle x \rangle \cap \langle ab \rangle=\langle x^2 \rangle$. 
Further we have by \cite[Lemma 3]{BW} that $C_P(x)=\langle x, ab \rangle$. 
Note that $C_P(x)=\langle x, ab \rangle=\langle x\rangle \langle ab \rangle$ has order equal to 
$$|x|\cdot |ab|/|\langle x \rangle \cap \langle ab \rangle|=|x|\cdot |ab|/|\langle x^2 \rangle|=2^{n+1}$$
since $x^2 \in Z(P)=\langle ab \rangle$. Note also that 
$Z(P)= \langle ab \rangle  \leq C_P(x) \cap P_0$. On the other hand, by the above order argument we have that 
$|C_P(x)\cap P_0| < 2^{n+1}$ since $x \notin P_0$. As a result, $Z(P)= \langle ab \rangle  = C_P(x) \cap P_0$.
\end{proof}

\begin{Lemma}\label{essential}
Let $G$ be a finite group with a $2$-subgroup 
$P\cong C_{2^n}\wr C_2$ for some $n\geq 2$. Assume that 
$\CF_P(G)$ is saturated,
and that 
the base subgroup $P_0$
is $\CF_P(G)$-essential. Further, let $Q:= \langle a^{2^m} \rangle \times \langle b^{2^m} \rangle $ 
where $0 \leq m \leq n-1$. Then, we have that $N_G(Q) = N_G(P_0) \, C_G(Q)$. 
Moreover, we have that $\Aut_{\CF}(Q) \cong S_3$.
\end{Lemma}

\begin{proof}
Set $\CF:=\CF_P(G)$.
Observe that $Q \, C_P(Q)=P_0$ and $N_P(Q)=P$, so 
${{}N_P(Q)/Q\,C_P(Q)} \cong C_2$.
Thus, there is an involutory automorphism of $Q$ which is induced from 
conjugation by an element in $P$.
Since $P_0$ is an $\CF$-essential subgroup, 
we have by {{}\cite[Proposition 6.12]{Sam14}} that $N_G(P_0)/C_G(P_0)\cong\Aut_{\CF}(P_0) \cong S_3$. 
This means that there is an automorphism of order 3 
of $P_0$ in $\CF$ that takes $a$ to $b$ and $b$ to $a^{-1} b^{-1}$ (so that $a^{-1} b^{-1}$ maps to $a$). 
This automorphism 
restricts to an automorphism of $Q$. So, we get that 
\begin{equation}\label{S3}    S_3 \leq \Aut_{\CF}(Q)\cong N_G(Q)/C_G(Q).
\end{equation}

{{}

Since $Q\unlhd P$, $Q$ is a fully $\CF$-normalized subgroup of $P$. So that
by Lemma \ref{Lin2.5},
$N_P(Q)/C_P(Q) \in{\mathrm{Syl}}_2(N_G(Q)/C_G(Q))$.
That is,
$|N_G(Q)/C_G(Q)|_2=2$ since we have known that $|N_P(Q)/C_P(Q)|=2$.
Obviously we can consider $N_G(Q)/C_G(Q) \leq{\mathrm{Aut}}(Q)$.
It is well-known that ${\mathrm{Aut}}(Q)$ is a $\{2,3\}$-group and that $|{\mathrm{Aut}}(Q)|_{3}=3$
(see e.g. \cite[line 1 p.5957]{CG}).
Hence by (\ref{S3}),
$$ N_G(Q)/C_G(Q)\cong {\mathrm{Aut}}_\CF(Q) \cong S_3.$$
}\noindent
So the second claim is established. 

To show that $N_G(Q) \leq N_G(P_0) \, C_G(Q)$, we claim fist that every 
automorphism in $\Aut_{\CF}(Q)$ extends to an automorphism in 
$\Aut_{\CF}(P_0) $. Suppose that $\alpha \in \Aut_{\CF}(Q) \cong S_3$. 
If $\alpha$ has order $2$, let's say $\alpha$ is conjugation induced by $t$, then $\alpha$ flips 
$ a^{2^m}$ and  $b^{2^m}$. Then, $\hat{\alpha}:P_0 \to P_0$ that flips $a$ and $b$, lies in $\Aut_{\CF}(P_0)$ 
and $\hat{\alpha} |_Q=\alpha$. If $\beta$ is an automorphism of $Q$ of order $3$, then 
we can choose $\beta$ as follows: $\beta$ takes $ a^{2^m}$ to $ b^{2^m}$ and $b^{2^m}$ to $ a^{-2^m}  b^{-2^m}$.  
Consider $\hat{\beta}:P_0\to P_0$ which takes takes $a$ to $b$ and $b$ to $a^{-1} b^{-1}$, 
then $\beta \in \Aut_{\CF}(P_0)$ and $\hat{\beta}|_Q= \beta$. So the claim on extension of automorphisms is established. 
Therefore, if $g \in N_G(Q)$ then there is an $h \in N_G(P_0)$ 
such that $c_h {\vert}_{Q} =c_g$ which implies that $h^{-1} g \in C_G(Q)$. So we have that 
$g \in N_G(P_0)\,C_G(Q)$. Now we claim that $N_G(Q) \geq N_G(P_0) \, C_G(Q)$. 
Suppose that $g \in N_G(P_0)$, then $c_g \in \Aut_{\CF}(P_0)$. Note that $Q=\Omega_{n-m}(P_0):=\langle g \in P_0 \ : \ g^{2^{n-m}}=1\rangle $,
so $Q$ is a characteristic subgroup of $P_0$. As a result, $c_g$ restricts to an $\CF$-automorphism of $Q$, in other 
words, $g\in N_G(Q)$. This finishes the first assertion.
\end{proof}

\begin{Lemma}\label{Q8}
Assume that $Q$ is a subgroup of $P$ isomorphic to $Q_8$. 
Then, $Q=\langle(ab^{-1})^{2^{n-2}}, a^i b^j t \rangle$ for some $i, \, j$ 
such that $i+j\equiv 2^{n-1} \ (\normalfont{\mathrm{mod}} \ 2^n)$. 
Moreover, $C_P(Q)=Z(P) (=\langle ab\rangle\cong C_{2^n})$. Furthermore, 
any two subgroups of $P$ which are isomorphic to $Q_8$ are $P$-conjugate.
\end{Lemma}

\begin{proof}
Since $Q$ is a non-abelian subgroup of $P$, $Q$ has 
a generator outside $\langle a \rangle \times \langle b \rangle $. Let $a^ib^j t$
be one of generators of $Q$ of order $4$ for some integers $i, j$. 
The element $(a^ib^j t)^2=(a b)^{i+j}$ is an involution. Hence
we must have that $i+j\equiv2^{n-1} \text{ (mod } 2^n)$. As a consequence, 
we get that the element $(a b)^{2^{n-1}}$ has to
be the unique involution of $Q$. 
Assume that $a^kb^\ell t$ is another generator of $Q$ of order $4$ for some integers $k$ and $\ell$,
then by the same reason,
we must have $k+\ell\equiv2^{n-1} \text{ (mod} \ 2^n)$. 
Moreover, by the defining relations of $Q_8$, the element
\begin{equation}\label{abt}
 (a^ib^j t) (a^kb^\ell t) =a^{i+\ell} b^{j+k}
 \end{equation}
must also have order $4$, and its square must be equal to the unique involution of $Q$, 
which is equal to $(a b)^{2^{n-1}}$.
Hence, we get the following system of equations:
$$i+j\equiv2^{n-1} \text{ (mod} \ 2^n)$$
$$k+\ell \equiv2^{n-1} \text{ (mod} \ 2^n)$$
$$2(i+\ell )\equiv2^{n-1} \text{ (mod} \ 2^n)$$
$$2(j+k)\equiv2^{n-1} \text{ (mod} \ 2^n).$$
This system gives the following two compatible system of equations. Either the following system
$$i+j\equiv2^{n-1} \text{ (mod} \ 2^n)$$
$$k+\ell\equiv2^{n-1} \text{ (mod} \ 2^n)$$
$$i+\ell\equiv2^{n-2} \text{ (mod} \ 2^n)$$
$$j+k\equiv2^{n-1}+2^{n-2} \text{ (mod} \ 2^n)$$
or the following system of equations
$$i+j\equiv2^{n-1} \text{ (mod} \ 2^n)$$
$$k+\ell\equiv2^{n-1} \text{ (mod} \ 2^n)$$
$$i+\ell\equiv2^{n-1}+ 2^{n-2} \text{ (mod} \ 2^n)$$
$$j+k\equiv2^{n-2} \text{ (mod} \ 2^n)$$
must be satisfied.

If $i, j, k, \ell$ satisfy the first system of equations, then from (\ref{abt}), we see that
$a^{2^{n-2}} b^{3.2^{n-2}}=(ab^{-1})^{2^{n-2}}\in Q$.
If the second system of equations is satisfied by $i, j, k, \ell$, then similarly 
$a^{3. 2^{n-2}} b^{2^{n-2}}=(a^{-1}b)^{2^{n-2}}\in Q$.
Note that these elements are of order $4$ and 
they are inverses of each other and their squares are equal to the
unique involution of $Q$. So $Q=\langle(ab^{-1})^{2^{n-2}}, a^i b^j t \rangle$ as desired.
Furthermore, since $C_P((ab^{-1})^{2^{n-2}})=P_0$, we conclude that
$C_P(Q)=C_P(a^i b^j t) \,\cap\,  C_P((ab^{-1})^{2^{n-2}})= \langle ab \rangle=Z(P)$ by using
Lemma \ref{CentralizerOutsideP0}. 

Let $\langle(ab^{-1})^{2^{n-2}}, a^{i_1} b^{j_1} t \rangle$ and
$\langle(ab^{-1})^{2^{n-2}}, a^{i_2} b^{j_2} t \rangle$ be subgroups of $P$ 
which are isomorphic to $Q_8$ so that
$i_1 + j_1\equiv i_2 + j_2 \equiv 2^{n-1} \ (\normalfont{\mathrm{mod}} \ 2^n)$. 
Then setting $x=a^{j_2} b^{j_1} t$, it follows that conjugation by $x$ sends first subgroup 
to the second. Hence, the last assertion is established.
\end{proof}

From Lemma \ref{Q8}, now it is easy to see that in a wreathed $2$-group $P$, the subgroups which 
are isomorphic to $Q_8 * C_{2^n}$ are built as follows: take a subgroup $Y$ in $P$ which is isomorphic to 
$Q_8$, then since the centralizer of $Y$ is $Z(P)$, this subroup should be equal to $Y \, Z(P)$ (see Theorem 2.5.3 in \cite{Gor}). 
So again using Lemma \ref{Q8}, we deduce that these subgroups are all $P$-conjugate to each other. From now on, let us fix one of 
them and call it $P_1$.

\begin{Lemma}\label{AutP1}
Suppose that $\CF$ is a saturated fusion system on $P$ and 
${\mathrm{Out}}_{\CF}(P_1)\cong S_3$. 
Then any $\CF$-automorphism of $P_1$ centralizes $Z(P_1)=Z(P)$. In particular, $\CF$ contains 
all possible automorphisms of $P_1$ which centralize $Z(P_1)$.
\end{Lemma}

\begin{proof}
Recall that $\Out_{\CF}(P_1)=\Aut_{\CF}(P_1)/ \Inn(P_1)$ where $\Inn(P_1)=P_1/Z(P_1)\cong C_2 \times C_2$. 
It is easy to see that any inner automorphism of $P_1$ centralizes $Z(P_1)$. Consider an 
outer automorphism of order $3$.
Note that using this outer automorphism, we can construct an actual automorphism $\alpha$
of $P_1$ of order $3$ since $\Inn(P_1)\cong C_2 \times C_2$. 
Then, since $Z(P)$ is a cyclic subgroup of $P_1$ which is characteristic, 
$\alpha$ centralizes the center of $P$ (see third paragraph in the proof of \cite[Proposition 2]{ABG}).  

Let $\beta$ be an arbitrary $\CF$-automorphism of $P_1$. Then since $\Out_{\CF}(P_1) \cong S_3$,
$\beta$ corresponds to an outer automorphism of order at most $2$ or it corresponds to 
an outer automorphism of order $3$. If $\beta$ corresponds to an outer automorphism of order $3$, 
then $\beta$ is a composition of an $\CF$-automorphism of order $3$ and an inner automorphism of $P_1$.
 Hence, by the first paragraph of the proof, $\beta$ centralizes $Z(P)$. If it corresponds to an outer automorphism 
of order at most $2$, it means that $\beta$ is a composition of an $\CF$-automorphism of order $3$ 
followed by conjugation induced by an element of $N_P(P_1)$ and clearly this also fixes all elements in the center $Z(P)$. 
\end{proof}

\begin{Lemma}\label{HomocyclicSubgroup}
Let $P$ be a wreathed $2$-group $C_{2^n}\wr C_2$ for some $n\geq 2$.
Assume that $Q\leq P$ and $Q\cong C_{2^m}\times C_{2^m}$ for some $m$
with $1\leq m\leq n$. 
\begin{enumerate}
\item[\rm  (i)] If $m\geq 2$, then $Q\leq P_0$ and
$Q=\langle a^{2^{n-m}}\rangle\times\langle b^{2^{n-m}}\rangle$.
\item[\rm (ii)] If $Q\cong C_2\times C_2$ and $Q\leq P_0$, then
$Q=\langle a^{2^{n-1}}\rangle\times\langle b^{2^{n-1}}\rangle$.
\item[\rm (iii)] If $Q\cong C_2\times C_2$ and $Q\,{\not\leq}\,P_0$, 
then 
we can assume that 
$Q=\langle (ab)^{2^{n-1}}\rangle\times\langle t\rangle$. Furthermore, $Q \leq_P P_1$. 
 \end{enumerate}

\end{Lemma}
\begin{proof}

(i) {\sf Suppose that $Q\,{\not\leq}\,P_0$.}
We can set $Q:=\langle x\rangle\times\langle y\rangle$ for some $x,y\in Q$ with $|x|=|y|=2^m$. 
We can assume that $x\,{\not\in}\,P_0$. Hence for some integers $r, s$ we may write
$x:=a^rb^st$. Obviously, $x^2=(ab)^{r+s}\in Z(P)$.

Assume first that $y\,{\not\in}\,P_0$. 
Then we can set $y=a^ib^jt$ for integers $i, j$, so that $y^2=(ab)^{i+j}\in Z(P)$.
Since $\langle x^2\rangle\cong\langle y^2\rangle\cong C_{2^{m-1}}\,{\not=}\,1$,
since $\langle x^2\rangle\leq Z(P)\geq \langle y^2\rangle$, and since $Z(P)$ is cyclic,
we must have $\langle x^2\rangle=\langle y^2\rangle$, a contradiction
since $Q=\langle x\rangle\times\langle y\rangle$.

Hence $y \in P_0$. Obviously $y\in C_P(x)$ since $Q=\langle x\rangle\times\langle y\rangle$,
so that $y\in C_P(x)\cap P_0$, and hence by Lemma  \ref{CentralizerOutsideP0}
$$y\in Z(P).$$
Thus, by noting that $\langle y\rangle\leq Z(P)\geq\langle x^2\rangle$, that $Z(P)$ is cyclic
and further that $|y|>|x^2|$, we must have that $\langle x^2\rangle <\langle y\rangle$,
which implies that $\langle x^2\rangle\cap\langle y\rangle=\langle x^2\rangle\cong C_{2^{m-1}}\,{\not=}\,1$
since $m\geq 2$. But this is {\sf a contradiction} since 
$\langle x^2\rangle\cap\langle y\rangle\leq\langle x\rangle\cap\langle y\rangle=1$.

Thus we have $Q\leq P_0$.
{
Next, notice
$Q = \{ \text{all the elements of order }2^m\text{ in }P_0\}$. Hence
$$Q=\{ a^ib^j\,|\, \ \ 2^{n-m}\,|\,i \ \text{ and } \    2^{n-m}\,|\,j  \}.$$
So, the second assertion holds.}

(ii) Since $Q$ has three different involutions in $P_0$, and $P_0$ has precisely three involutions
$a^{2^{n-1}}, b^{2^{n-1}}, a^{2^{n-1}}{\cdot}b^{2^{n-1}}$, the assertion follows.

(iii) From \cite[Lemma 2 (iv)]{ABG}, we can conclude that any such $Q$ is $P$-conjugate to 
$\langle (ab)^{2^{n-1}}\rangle\times\langle t\rangle$. Note that 
$\langle (ab)^{2^{n-1}}\rangle\times\langle t\rangle \leq \langle (ab^{-1})^{2^{n-2}}, (ab)^{2^{n-2}} t \rangle \, Z(P) =_P P_1$, 
so it follows that $Q \leq_P P_1$ in this case.
\end{proof}

\begin{Lemma}\label{essential2}
Let $G$ be a finite group with a $2$-subgroup 
$P\cong C_{2^n}\wr C_2$ for some $n\geq 2$. Assume that 
$\CF_P(G)$ is saturated,
and $P_1$
is $\CF_P(G)$-essential. Further, let $Q\leq P_1$ such that $Q \cong Q_8 * C_{2^m}$
where $0 \leq m \leq n-1$. Then, we have that $N_G(Q) = N_G(P_1) \, C_G(Q)$. 
Moreover, we have that ${\mathrm{Out}}_{\CF}(Q) \cong S_3$. 
On the other hand, if $Q \leq P_1$ 
is an $\CF_P(G)$-normalized subgroup such that $Q \cong C_2 \times C_2$, 
we have that ${\mathrm{Out}}_{\CF}(Q) \cong C_2$.
\end{Lemma}

\begin{proof}
Suppose that $Q \cong Q_8 * C_{2^m}$ for $1 \leq m \leq {{}n}$.
Then  by Lemma \ref{Q8} and  \cite[Theorem 2.5.3]{Gor}),
\begin{align*}
Q&=Y\,K=Y*K \text{ where}
\\
Y=\langle(ab^{-1})^{2^{n-2}}, a^i b^j t \rangle &\cong Q_8 \text{ and }
K :=\langle (ab)^{2^{n-m}}\rangle\leq Z(P)\text{ and }K\cong C_{2^m}
\\
\text{for some }&i, \ j \text{ with }i+j\equiv 2^{n-1} (\text{ mod } 2^n). 
\end{align*}

Set $\CF:=\CF_P(G)$ and let us choose $P_1=Y \, Z(P)$ so that $Q \leq P_1$. By the definition of $Y$, we have that 
\begin{align}
\label{Y}
 Y=\{&1, (ab^{-1})^{2^{n-2}}, (ab)^{2^{n-1}}, (a^{-1}b)^{2^{n-2}}, 
a^ib^jt, (ab^{-1})^{2^{n-2}} a^ib^j t, (ab)^{2^{n-1}} a^ib^j t, (a^{-1}b)^{2^{n-2}} a^ib^j t\}.
 \end{align}
From Lemma \ref{Q8}, we have that
$Q \, C_P(Q)=Y \, Z(P)$, so by looking at the elements of $Y$, we deduce 
that any element  of the form $a^k b^l t$ in $Q \, C_P(Q)$ satisfy either
\begin{align}
\label{kl}
k-l=i-j \text{ or } k-l=i-j+2^{n-1}.
\end{align}
{\sf Next we claim that} $N_P(Q)/Q\,C_P(Q)\,{\not=}\,1$. 
For this, it suffices to prove that there exists an
element $x$ in $N_P(Q) - Q\,C_P(Q)$.
To see this, let $r$ and $s$ be integers satisfying
$$
r-s\equiv i-j+2^{n-2} \ (\text{mod } 2^n).
$$
and let $x=a^rb^st\in P$. 
Noting that
$$
x(ab^{-1})^{2^{n-2}}x^{-1}=(a^rb^st) (ab^{-1})^{2^{n-2}}(a^{-s}b^{-r}t)=(a^{-1}b)^{2^{n-2}}
$$
$$
x (a^ib^jt)x^{-1}=(a^rb^st) (a^ib^jt)(a^{-s}b^{-r}t)=(ab^{-1})^{2^{n-2}} (a^ib^jt).
$$
So, $x\in N_P(Y)$ which implies that $x\in N_P(Q)$. Further, we know by (\ref{kl}) that $x\,{\not\in}Q\,C_P(Q)$.
Hence $x\ Q\,C_P(Q)$ is a non-trivial element 
in $N_P(Q)/Q\,C_P(Q)$, that is $N_P(Q)/Q\,C_P(Q)\,{\not=}\,1$, {\sf thus the claim is proved.}  
Now, we have that $P_1=Y\,Z(P) \geq Y\,K=Q\cong Q_8*C_{2^m}$. First, note that $N_P(Q)\leq N_P(P_1)$.
Obviously, $C_P(P_1)\leq C_P(Y)$ and $C_P(Y)=Z(P)$ by Lemma \ref{Q8} (don't confuse that
$Y$ here is $Q$ in Lemma \ref{Q8}). So that $C_P(P_1)\leq Z(P)\leq P_1$.
Further, by \cite[Theorem 1.3(2)]{CG}, ${\mathrm{Aut}}(P_1)$ is not a $2$-group. Thus,
we can apply \cite[Lemma 3(ii), p.10]{ABG} to the group $P_1\leq P$, namely we have that
$$  N_P(P_1)/P_1 \cong C_2.$$
It is easy to know that $Q\unlhd P_1$, so that $P_1\leq N_P(Q)$.
Further $P_1=Y\,Z(P)=Q\,C_P(Q)=Q\,C_P(Y)$, so that $P_1 \unlhd N_P(Q)$. Thus, 
$$ 
1\,{\not=}\,N_P(Q)/Q\,C_P(Q)=N_P(Q)/P_1 \leq N_P(P_1)/P_1 \cong C_2.
$$
Hence $N_P(Q)/Q\,C_P(Q)\cong C_2$.

On the other hand, since $P_1$ is $\CF$-essential, it follows from 
\cite[Proposition 2]{ABG} that  
$N_G(P_1)/P_1 \, C_G(P_1)\cong\Out_{\CF}(P_1) \cong S_3$. 
By Lemma \ref{AutP1}, any $\CF$-automorphism of $P_1$ centralizes $Z(P)$. 
Take an automorphism $\alpha$ of $P_1$ of order $3$.
Then $\alpha$ 
restricts to an automorphism of $Q$, since $P_1/Z(P_1)=Y/Z(Y)=Q/Z(Q)$. So, we get that 
\begin{equation}\label{S33}    S_3 \leq \Out_{\CF}(Q)\cong N_G(Q)/Q \,C_G(Q).
\end{equation}
Note that all subgroups which are isomorphic to $Q$ is $P$-conjugate to $Q$, so $Q$ is 
a fully $\CF$-normalized subgroup of $P$. So it follows 
from Lemma \ref{Lin2.5} that
$N_P(Q)/Q \, C_P(Q) \in{\mathrm{Syl}}_2(N_G(Q)/Q \,C_G(Q))$.
That is,
$|N_G(Q)/Q\,C_G(Q)|_2=2$ since we have known that $|N_P(Q)/Q\;C_P(Q)|=2$.
Obviously we can consider $N_G(Q)/Q\,C_G(Q) \leq{\mathrm{Out}}(Q)$.
It is well-known that ${\mathrm{Out}}(Q)$ is a $\{2,3\}$-group and that $|{\mathrm{Aut}}(Q)|_{3}=3$
(see e.g. \cite[Theorem 6.2]{CG}).
Hence by (\ref{S33}),
$$ N_G(Q)/Q\,C_G(Q)\cong {\mathrm{Out}}_\CF(Q) \cong S_3.$$
So the second assertion is established for $Q \cong Q_8 * C_{2^m}$. 

To show that $N_G(Q)=N_G(P_1)\, C_G(Q)$, we first claim that any element
in $\Aut_{\CF}(Q)$ can be induced to an element in $\Aut_{\CF}(P_1)$. 
Recall that $Q=Y \, K$ where $Y\cong Q_8$ and $K\leq Z(P)$. Notice that $Z(Q)=K \leq Z(P)$. 
We first claim that any $\CF$-automorphism of $Q$ centralizes $Z(Q)$. It is easy to see that 
any automorphism in $\Aut_P(Q)$ centralizes $Z(Q)$. So let us assume that $\alpha \in \Aut_{\CF}(Q)$ 
has order which is not equal to a power of $2$. Then the image of $\alpha$ in $\Out_{\CF}(Q)$ has order 
$3$, or equivalently $\alpha^3 \in \Inn(Q)\cong C_2 \times C_2$. So either $\alpha^3$ is identity or $\alpha^3$ 
is an involution. In the former case, $\alpha$ itself has order $3$ and $Z(Q)$ is a cyclic characteristic subgroup 
of $Q$, so $\alpha$ centralizes $Z(Q)$ in this case. Now assume that $\alpha^3$ is an involution, then $\alpha^2$ 
is an $\CF$-automorphism of $Q$ which has order $3$. By the same argument as in the previous case, 
we reach that $\alpha^2$ centralizes $Z(Q)$. Note also that since $\alpha^3$ is an inner automorphism 
of $Q$, it also centralizes $Z(Q)$. As a result, $\alpha$ centralizes $Z(Q)$. Hence, any $\CF$-automorphism 
of $Q$ fixes all elements of $Z(Q)$ and so corresponds to a permutation of the three non-trivial elements of $Q/Z(Q)$. 
Now, let $\varphi \in \Aut_{\CF}(Q)$ be an arbitrary automorphism, let also $\sigma_{\varphi}$ denote the 
permutation of $Q/Z(Q)$ which $\varphi$ corresponds. We will construct $\hat{\varphi}: P_1 \to P_1$ 
such that $\hat{\varphi}$ centralizes the center of $P_1$ and $\hat{\varphi}$ does the same move on 
$P_1/Z(P_1)$ as $\sigma_{\varphi}$ does to $Q/Z(Q)$, since $P_1/Z(P_1) = Y/Z(Y)=Q/Z(Q)$. Then 
from Lemma \ref{AutP1}, we deduce that $\hat{\varphi} \in \Aut_{\CF}(P_1)$ and it follows that 
$\hat{\varphi} |_Q=\varphi$.
Therefore, every automorphism in $\Aut_{\CF}(Q)$ extends to an automorphism in 
$\Aut_{\CF}(P_1) $. Consequently, if $g \in N_G(Q)$ then there is an $h \in N_G(P_1)$ 
such that $c_h {\vert}_{Q} =c_g$ which implies that $h^{-1} g \in C_G(Q)$. So we have that 
$g \in N_G(P_1)\,C_G(Q)$. 
Conversely, let $g \in N_G(P_1)$, then $c_g:P_1 \to P_1$ is in $\Aut_{\CF}(P_1)$. From 
Lemma \ref{AutP1}, it follows that $c_g$ centralizes center of $P_1$ and so it corresponds to a permutation 
of $P_1/Z(P_1)$. Since $P_1/Z(P_1) = Y/Z(Y)=Q/Z(Q)$, $c_g$ corresponds to the same permutation 
of $Q / Z(Q)$ and it fixes all elements in $Z(Q)$. So $c_g$ is an $\CF$-automorphism of $Q$ since restriction 
of all homomorphisms in $\CF$ lies in $\CF$. In other words $c_g \in N_G(Q)$. 

Suppose that $Q \cong C_2 \times C_2$, then by Lemma \ref{HomocyclicSubgroup}(iii), 
we have that $Q=\langle (a b)^{2^{n-1}} \rangle \times \langle t \rangle$. In this case, 
$Q \, C_P(Q)=Z(P) \times \langle t \rangle$ and 
$$N_P(Q)=\langle (ab^{-1})^{2^{n-2}}, (ab)^{2^{n-2}} t\rangle \, Z(P)=_P P_1,$$
so that $N_P(Q)/Q \,C_P(Q) \cong C_2$. We claim that $\Aut_{\CF}(Q)$ does not include 
an automorphism of order $3$. Arguing by contradiction, suppose that $\varphi \in \Aut_{\CF}(Q)$ 
is an automorphism of order $3$, which would imply that $\Aut_{\CF}(Q)\cong S_3$. Then 
$O_2(\Aut_{\CF}(Q))$ is trivial, so since $Q$ is assumed to be fully $\CF$-normalized, 
Lemma 5.4 of \cite{L} implies that $\varphi$ extends to 
some $\hat{\varphi} \in \Aut_{\CF}(Q \, C_P(Q))$. However, from Corollary 2.3(4) of \cite{CG}, it 
follows that $\Aut(Q \, C_P(Q)) \cong \Aut(C_{2^n} \times C_2)$ is a $2$-group since $n \geq 2$. 
So we arrive to a contradiction and we have that $\Aut_{\CF}(Q)=\Out_{\CF}(Q) \cong C_2$.
\end{proof}

\begin{Lemma}\label{list}
Let $G$ be a finite group with a wreathed $2$-subgroup $P:=C_{2^n}\wr C_2$
for $n\geq 2$. 
Suppose that $\CF_P(G)$ is saturated. 
Assume that $Q\leq P$ is an $\CF_P(G)$-normalized subgroup such that $N_G(Q)/Q\,C_G(Q)$ is not a $2$-group.
Then, $Q$ is equal to one of the groups in the following list:
\begin{enumerate}
\item[$\bullet$] $\langle a^{2^m} \rangle \times \langle b^{2^m}\rangle\cong C_{2^{n-m}} \times C_{2^{n-m}} $ for some $0 \leq m \leq n-1$, so that $Q \leq P_0$;
\item[$\bullet$]  $Y \, K \cong Q_8 * C_{2^m}$ for some $1 \leq m \leq n$ where $Y \cong Q_8$ and $K \leq Z(P)$ and in this case $Q \leq_P P_1$ (note that $Q_8 * C_{2} \cong Q_8$).
\end{enumerate}

\end{Lemma}

\begin{proof}
Since $P$ has $2$-rank $2$, the $2$-rank of any subgroup of $P$ is less than or equal to $2$. So the $2$-rank 
of $Q$ is either $1$ or $2$. A $2$-group with rank $1$ is either isomorphic to a cyclic or a generalized 
quaternion group
 Among these only $Q_8$ has automorphism group which is not a $2$-group
{{} (see \cite[Lemma 2.1(a)]{O})}.

Now, assume that $Q$ has $2$-rank $2$. Then \cite[Theorem 1.3]{CG} 
gives the list of all $2$-groups of $2$-rank $2$ whose
automorphism groups are not $2$-groups. 
We will go over the list and decide which groups in the list can occur as subgroups of $P$. 

In a wreathed $2$-group there is no subgroup isomorphic
to either $Q_8 \times C_{2^m}$ for $m\geq 1$ or $Q_8 \times Q_{2^t}$ for $t\geq 3$. 
Indeed, from Lemma \ref{Q8}, we know that any subgroup of $P$
which is isomorphic to $Q_8$ has its centralizer in $P$ equal to $Z(P)=\langle ab \rangle$, 
so a generator of $C_{2^m}$ or $Q_{2^t}$ 
should be of the form $(ab)^{2^k}$ for some $k$ where $0\leq k \leq n-1$. 
But this would imply that $((ab)^{2^k})^{2^{n-1-k}}=(ab)^{2^{n-1}}$ 
to lie in either of these groups. However, this can not happen since this element lies in every subgroup of $P$ which is isomorphic to $Q_8$ by Lemma \ref{Q8}. 
Hence, we cannot construct such direct products inside $P$. 

Now, let us consider the central product group case. In this case, there are two isomorphism classes
of groups in the list: $Q_8 * C_{2^m}$ for $m\geq 2$ and $Q_8 * D_{2^t}$ for $t\geq 3$. 
By \cite[Theorem 2.5.3]{Gor}, to form a central product $Q_8 * K$,
the group which is isomorphic to $Q_8$ should centralize $K$, so that $K \leq Z(P)$ by Lemma \ref{Q8}. Since $Z(P)$ is
cyclic, we can only have $Q_8 * C_{2^m}$ among the possible central products. 

From the constructions of $X_n$ and $Y_n$,
it is easy to see that they can not occur as a subgroup of $P$ as they should contain a subgroup isomorphic to $Q_8 \times C_{2^m}$
\cite[Section 6, third paragraph]{CG} which can not be a subgroup of $P$ by the argument in the third paragraph of the proof. 
The wreath product $Q_8 \wr C_2$ also can not occur as a subgroup of $P$, since $Q_8 \times Q_8$ 
can not occur by the third paragraph of the 
proof, similarly. 

The final candidate in the list is the Sylow $2$-subgorup
of ${\mathrm{U}}_3(4) ={\mathrm{PSU}}_3(4)$
which is a Suzuki $2$-group. Let us call it {\sf Suz}. 
By \cite[(3) on p.50]{Col72},
$Z({\sf Suz}) \cong C_2 \times C_2$ 
however $Z(P)=\langle ab \rangle \cong C_{2^n}$. This is a contradiction 
since \cite[Lemma 2 (xii)]{ABG} states that every non-abelian subgroup of $P$ has
its center contained in the center of $P$.

It follows from Lemma \ref{essential} and Lemma \ref{essential2}, the subgroups listed in the 
statement are the only subgroups of $P$, whose $\CF$-automorphism group contain an odd order element. 
So the proof is established.
\end{proof}

\begin{Lemma}\label{SS3}
Let $P$ be a wreathed $2$-subgroup of $G$ and assume that $\CF:=\CF_P(G)$ is saturated. 
Assume that for $Q\leq P$,
$Q$ is isomorphic to one of the groups in Lemma \ref{list} and Q is fully $\CF_P(G)$-normalized.
If $N_G(Q)/Q\,C_G(Q) $ is not a $2$-group, then $N_P(Q)/Q\,C_P(Q)\cong C_2$ and 
$N_G(Q)/Q\,C_G(Q) \cong S_3$. 
\end{Lemma}

\begin{proof}
The classification of saturated fusion systems on $P$ given in \cite[Proposition 2]{ABG} or \cite[Theorem 5.3]{CG}, 
states that there are four possible saturated fusion systems which are determined by $\CF$-essentialities of $P_0$ and $P_1$.

{\sf Case 1:}
Let us assume that $Q=\langle a^{2^{m}} \rangle \times \langle b^{2^{m}} \rangle $
for some $m$ with $0 \leq m \leq n-1$. If $P_0$ is not $\CF$-essential, 
then the central involution $(ab)^{2^{n-1}}$ is not fused under $\CF$ to either $a^{2^{n-1}}$ 
or $b^{2^{n-1}}$, thus $\Aut_{\CF}(Q)$ is a $2$-group, in this case. If $P_0$ is $\CF$-essential, then the result follows 
from  Lemma \ref{essential}. 

{\sf Case 2:}
Suppose that $Q \cong Q_8 * C_{2^m}$ for $1 \leq m \leq {{}n}$.
Then we can assume that $Q \leq P_1$. If $P_1$ is not $\CF$-essential, 
then $\Aut_{\CF}(P_1)$ is a $2$-group which implies that $\Aut_{\CF}(Q)$ is a $2$-group 
because otherwise the odd order $\CF$-automorphism of $Q$ would extend to an odd order 
automorphism of $P_1$ (see the proof of Lemma \ref{essential2}). So assume that $P_1$ is $\CF$-essential, 
then the result follows from Lemma \ref{essential2}. 
\end{proof}

\begin{Theorem}\label{HQ}
Let $P$ be a wreathed $2$-subgroup of $G$, namely
$P\cong C_{2^n}\wr C_2$ for some $n\geq 2$,
and assume that $\CF_P(G)$ is a saturated fusion system. 
Let $Q\leq P$. Assume that
$Q$ is one of the groups in Lemma \ref{list} and Q is fully $\CF_P(G)$-normalized.
Assume, moreover, that $C_G(Q)$ is $2$-nilpotent and $N_G(Q)/Q\,C_G(Q) $ is not a $2$-group.
Then there exists a subgroup $H_Q$ of $N_G(Q)$ such that $N_P(Q)$ is a Sylow $2$-subgroup
of $H_Q$ and $|N_G(Q) : H_Q|$ is a power of $2$ (possibly $1$).
\end{Theorem}

\begin{proof}
Since $C_G(Q)$ is $2$-nilpotent, the group $Q\,C_G(Q)$ is also $2$-nilpotent. Let 
$K_Q:=O_{2'}(Q\,C_G(Q))$ and let $R_Q \in {\mathrm{Syl}}_2(Q\,C_G(Q))$ containing
$Q\,C_P(Q)$, so that $Q\,C_G(Q)=K_Q\rtimes R_Q$. Note that since $O_{2'}(Q\,C_G(Q)) 
=O_{2'}(C_G(Q))$ we have $[K_Q,Q]=1$ and so $K_Q\rtimes Q=K_Q \times Q$.
Note that $(K_Q\times Q) \unlhd (K_Q \rtimes N_P(Q))$. 
Moreover, since $K_Q$ is a characteristic subgroup of $Q\,C_G(Q)$ and $Q\,C_G(Q) \unlhd N_G(Q)$,
we have $K_Q \unlhd N_G(Q)$, so that $K_Q \times Q \unlhd N_G(Q)$.

We will take quotients with respect to $L_Q:=K_Q \times Q$ and use the notation 
$\overline{H}$ to denote the image of $H \leq N_G(Q)$
under the natural epimorphism $\pi_{L_Q}:N_G(Q) \twoheadrightarrow N_G(Q)/L_Q$. Then 
$\overline{Q\,C_G(Q)} \cong R_Q/Q$ is
a normal 2-subgroup of $\overline{N_G(Q)} $. 

Let us set $Q=P_0$ or $Q=P_1$. Then, by our assumption, it follows that $Q$ is $\CF$-essential, 
so \cite[Proposition 2]{ABG} implies that
$$\overline{N_G(Q)} \ / \ \overline{Q\,C_G(Q)} \cong N_G(Q) / Q\,C_G(Q) \cong S_3.$$
Since $Q$ is $\CF$-centric, it follows that $Q=Q \, C_P(Q)$ and so Lemma \ref{SS3} implies that 
$$\overline{K_Q \rtimes N_P(Q)} \cong N_P(Q)/Q 
=
N_P(Q)/ Q \, C_P(Q) \cong C_2.$$ 
It follows that there exists an involution $\bar x$ such that
$$
\bar x\in\overline{K \rtimes N_P(Q)}   -  \overline{K\rtimes Q\,C_P(Q)}
$$
(recall that $\overline{K \rtimes N_P(Q)}\cong N_P(Q)/Q$ and 
$\overline{K\rtimes Q\,C_P(Q)}\cong (Q\,C_P(Q))/Q)$. 
Hence, by Lemma \ref{BaerSuzuki} there is a subgroup 
$H$ of $\overline{N_G(Q)}$ such that ${{}\bar x} \in H \cong S_3$. Set $H_Q$
as the preimage of $H$ under $\pi_{L_Q}$. Then $H_Q$ contains $N_P(Q)$ as a Sylow $2$-subgroup 
since $H_Q$ contains $\langle x, Q \rangle=N_P(Q)$ by construction,
and that $|N_G(Q): H_Q|$ is a power of $2$ for $Q=P_0$ and $Q=P_1$.

Suppose now that $Q \neq P_i$ for $i=0, 1$, then Lemma \ref{list} implies that either $Q< P_0$ or $Q< P_1$, 
that is $Q <P_i$ where $i$ is equal to precisely one of $0$ or $1$. Moreover, if $Q <P_i$, it would imply that 
$P_i$ has an odd order automorphism and so $P_i$ is $\CF$-essential. 
Set $H_Q:= H_{P_i} \cdot K_Q$, then Lemma \ref{essential} and Lemma \ref{essential2} imply that 
$H_Q$ is a subgroup of $N_G(Q)$. We claim 
that $H_Q$ satisfies the required properties. To do this, first let us recall that 
$Q \, C_P(Q)=P_i$ and so from 
{\color{black}
Lemmas \ref{essential} and \ref{essential2}
}
  that $N_P(Q)=N_P(P_i).$ 
  
Now, we claim that $N_G(P_i) \cap (Q \, C_G(Q))=P_i \, C_G(P_i)$. Since $Q < P_i$, we have that 
$C_G(P_i) \leq C_G(Q)$ so that $P_i \,C_G(P_i) \leq P_i\, C_G(Q) = Q \, C_G(Q)$. As a result, 
$P_i C_G(P_i) \leq N_G(P_i) \cap Q \,C_G(Q)$. 

{\color{black}
Conversely, let $g \in N_G(P_i) \cap Q \,C_G(Q)$, then $c_g \in \Inn(Q)$, and from the proofs of 
Lemmas \ref{essential} and \ref{essential2}, it follows that $c_g \in \Inn(P_i)$, or equivalently 
$g\in P_i \, C_G(P_i)$.   }

Lastly, {\color{black}  $P_i \,C_G(P_i) \leq Q \, C_G(Q)$ 
and any element of $K_{P_i}$ has odd order, it follows that
$K_{P_i}$ is contained in $K_Q$ because} $K_Q$ contains all odd order elements of $Q \, C_G(Q)$.
Also, the order of $R_{P_i}$ divides the order of $R_Q$ and the result is a power of $2$. 
Now since $N_G(Q)=N_G(P_i) \ (Q \, C_G(Q))$, it follows that 
the order of $N_G(Q)$ is equal to 
$$ \frac{|N_G(P_i)| \cdot |K_Q| \cdot |R_Q|}{|K_{P_i}| \cdot |R_{P_i}|}.$$ 
On the other hand, let us observe now that $H_{P_i} \cap K_{Q} \leq N_G(P_i) 
\cap K_Q \leq N_G(P_i) \cap (Q \, C_G(Q))=P_i \, C_G(P_i)$, 
which implies that $H_{P_i} \cap K_{Q} =K_{P_i}$. So 
the order of $H_Q$ is equal to 
$$ \frac{|H_{P_i}| \cdot |K_Q| }{|K_{P_i} |}.$$ 
Hence, the index of $H_Q$ in $N_G(Q)$ is equal to $|R_Q|/|R_{P_i}|$ which is a 
power of $2$. Moreover, since we have that $N_P(Q)=N_P(P_i)$, we have that 
$N_P(Q)$ is a subgroup of $H_Q$, moreover from the order computation of $H_Q$ above, 
it is easy to see that $N_P(Q)$ is a Sylow $2$-subgroup of $H_Q$. Therefore, we have created 
the required $H_Q$ for all possible cases, and the proof is finished.
\end{proof}

\section{Scott modules with wreathed $2$-group vertices}
\noindent
In this section we investigate the Scott modules when their vertices are
the wreathed $2$-groups.

\begin{Lemma}\label{2-nilpotent}
Assume that $P\cong C_{2^n}\wr C_2$ for some $n\geq 2$ and $P$ is a 
subgroup of $G$ with
$\mathcal F:= \mathcal F_P(G)$ is saturated. 
Further if $Q$ is a fully $\mathcal F$-normalized subgroup of $P$ such that
$C_G(Q)$ is $2$-nilpotent, then 
$\Res\,^{N_G(Q)}_{Q\,C_G(Q)}\Big({\mathrm{Sc}} (N_G(Q), N_P(Q))\Big)$ is
indecomposable.
\end{Lemma}

\begin{proof}
Suppose that $Q$ is not equal to one of the groups in Lemma \ref{list}, then $N_G(Q)/C_G(Q)$ is a $2$-group.
Thus $N_G(Q)$ is $2$-nilpotent since $C_G(Q)$ is assumed to be $2$-nilpotent. So
we can write $N_G(Q)=K \rtimes S$ where $K:=O_{2'}(N_G(Q))$ and $S\in{\mathrm{Syl}}_2(N_G(Q))$.
Since $N_P(Q)$ is a $2$-subgroup of $N_G(Q)$, without loss of generality we can assume that $N_P(Q) \leq S$.
Set $H_Q:= K\rtimes N_P(Q)$, then $N_P(Q)\in{\mathrm{Syl}}_2(H_Q)$ and $|N_G(Q):H_Q|$ is
a power of $2$.

Suppose that $Q$ is equal to one of the groups in Lemma \ref{list}. 
If $N_G(Q)/Q\,C_G(Q)$ is a $2$-group, then $N_G(Q)/C_G(Q)$ is also a $2$-group, so that
we again get the desired subgroup $H_Q$ as in the previous paragraph.
So we can assume that $N_G(Q)/Q\,C_G(Q)$ is not a $2$-group.
In this case, Theorem \ref{HQ} implies that the desired subgroup $H_Q$ exists.

For all possible fully ${\mathcal F}$-normalized $Q$, we find the subgroup $H_Q$ satisfying the conditions of
Theorem \ref{IK2}, hence $\Res\,^{N_G(Q)}_{Q\,C_G(Q)}
\Big({\mathrm{Sc}}(N_G(Q), N_P(Q))\Big)$ is indecomposable.
\end{proof}

\begin{Theorem}\label{thm:M1}
Suppose that a finite group $G$ has a subgroup $P\cong C_{2^n}\wr C_2$ for some $n\geq 2$.
Assume further that the fusion system ${\mathcal F}_P(G)$ of $G$ over $P$ is saturated and that
$C_G(Q)$ is $2$-nilpotent for every fully ${\mathcal F}_P(G)$-normalized non-trivial subgroup
$Q$ of $P$. Then the Scott module
${\mathrm{Sc}}(G,P)$ is Brauer indecomposable.
\end{Theorem}

\begin{proof}
Follows from Lemma \ref{2-nilpotent} together with Theorem \ref{IK1}.
\end{proof}

\begin{Lemma}\label{C_G(Q)2-nilp}
Suppose that $P\in{\mathrm{Syl}}_2(G)$ such that $P\cong C_{2^n}\wr C_2$ for
some integer $n\geq 2$ and that $Q\leq P$ such that $Q$ is fully $\CF_P(G)$-normalized in $P$.
Let $P_0$
be the base subgroup of $P$.
If furthermore $Q$ satisfies one of the following five conditions, then $C_G(Q)$ is $2$-nilpotent;
\begin{enumerate}

\item\label{nonabelian} 
$Q$ is non-abelian.
\item
\label{non-homocyclic} 
$Q$ is abelian of $2$-rank $2$, non-homocyclic and $Q\,{\not\leq}\,P_0$.
\item
\label{cyclic} 
$Q$ is cyclic and $Q\,{\not\leq}\,P_0$.
\item
\label{2abelianLemma}
 $Q=P_0$ or $Q=Z(P) \times \langle t \rangle$.
 \item
 \label{hcyclic}
 $Q=\langle (ab)^{2^{n-1}} \rangle \times \langle t \rangle$.
\end{enumerate}
\end{Lemma}
\begin{proof}
Set $\CF:=\CF_P(G)$.
Since $P\in{\mathrm{Syl}}_2(G)$, $\CF$ is saturated by \cite[Proposition 1.3]{BLO}.
Since $Q$ is fully $\CF$-normalized in $P$,
we know by Lemma \ref{Lin2.5} that $Q$ is fully $\CF$-centralized in $P$, so that 
from Lemma 2.10(i) of \cite{L}, $C_P(Q)\in{\mathrm{Syl}}_2(C_G(Q))$.

{\sf Assume Case (1):} 
Since $Q$ is a non-abelian subgroup of $P$, there exists an element 
$x\in Q$ and $x\,{\not\in}\,P_0:=\langle a\rangle\times\langle b\rangle$. 
Then $C_P(x)=\langle ab, x \rangle$ 
from \cite[Lemma 3]{BW}.  Obviously, $C_P(x)$ is abelian since 
$\langle ab\rangle=Z(P)$ and $C_P(x)=\langle x\rangle\,Z(P)$.
Hence $Q\,{\not\leq}\,C_P(x)$ since $Q$ is non-abelian.
Thus, there is an element $y\in Q  -  C_P(x)$.
Set $C:=C_P(x)\cap C_P(y)$. 

Now {\sf we claim} that 
$$C_P(Q) \leq C = Z(P).$$
The first inequality is clear. Since $x \not\in P_0$, 
Lemma \ref{CentralizerOutsideP0} implies that $|C_P(x)|=2^{n+1}$.
Since $[x,y] {\not=}1$, $x\in C_P(x) -  C$, so that $C \lneqq C_P(x)$,
and hence $|C| < 2^{n+1}$.
Clearly $Z(P)\leq C$. Then since $|Z(P)|=2^n$, we get that $|C|=2^n$,
which implies that $Z(P)=C$. {\sf The claim} is proved. 
   
Since $Z(P)$ is cyclic, $C_P(Q)$ is cyclic. 
so by the Burnside normal $p$-complement theorem 
\cite[Theorem 7.6.1]{Gor},  $C_G(Q)$ is $2$-nilpotent.

{\sf Assume Case (2):}
Since $Q$ is abelian, non-homocyclic and $2$-rank $2$, 
$Q=\langle x \rangle \times \langle y \rangle$ for some $x, y \in P$  with $|x|\neq |y|$. 
Since $Q\,{\not\leq}\, P_0:=\langle a \rangle \times \langle b \rangle$, 
we can assume that $x=a^k b^\ell t$ for some integers 
$k, \ell$. Then $x^2=(ab)^{k+\ell}\in Z(P)$. 

\indent\indent
Assume first that $y\,{\not\in}\,P_0$. Then $y^2 \in Z(P)$ as above. 
Since $\langle x \rangle \cap \langle y \rangle=1$ and $\langle x^2 \rangle \cap \langle y^2 \rangle \leq \langle x \rangle \cap \langle y \rangle$,
we have that $\langle x^2 \rangle \cap \langle y^2 \rangle=1$. 
Since $Z(P)$ is cyclic and $x^2, y^2 \in Z(P)$, 
this can happen when either $x$ or $y$ is an involution.
Furthermore, by Lemma 3 of \cite{BW}, $C_P(x)=\langle x, ab \rangle $ and similarly 
$C_P(y)=\langle y, ab \rangle $, where $x\in C_P(y)$ and $y \in C_P(x)$. So $C_P(x)=C_P(y)$.
Hence $C_P(Q)=C_P(x)=C_P(y)$ since $Q = \langle x, y\rangle$.
Thus, we have that 
$C_P(Q)=\langle x \rangle \times \langle ab \rangle $ if $x$ is an involution,
and that $C_P(Q)=C_P(y)=\langle y \rangle \times \langle ab \rangle $ if $y$ is an involution. 
In both cases, we have that $C_P(Q)\cong C_2 \times C_{2^n}$. 

\indent\indent
Next assume that $y\,\in\,P_0$.  Then 
$y \in C_P(x)\cap P_0$. And from Lemma \ref{CentralizerOutsideP0}, we have that $y\in Z(P)$.
Moreover, $\langle x^2 \rangle \cap \langle y \rangle \leq\langle x \rangle \cap \langle y \rangle=1$ 
implies that $\langle x^2 \rangle \cap \langle y \rangle=1$.
Since $Z(P)$ is cyclic and $x^2, y \in Z(P)$, this can happen only when 
$x$ is an involution.
So $C_P(Q)=C_P(x) \cap C_P(y)= \langle x, ab \rangle \cap P=\langle x, ab \rangle$.
Hence, in this case, we have that 
$C_P(Q)=\langle x \rangle \times \langle ab \rangle \cong C_2 \times C_{2^n}$
since $|x|=2$ and $x\,{\not\in}\,Z(P)$.

We have known already that 
$C_P(Q)\in{\mathrm{Syl}}_2(C_G(Q))$. 
Further, \cite[Corollary 2.3(4)]{CG} implies that
 $\Aut(C_P(Q)) \cong \Aut(C_{2}\times C_{2^n})$ is a 
$2$-group since $n \geq 2$. So the normalizer of $C_P(Q)$ 
in $C_G(Q)$ is equal to the centralizer of $C_P(Q)$ in $C_G(Q)$
since $C_P(Q)$ is abelian.
Therefore, Burnside normal $p$-complement theorem 
\cite[Theorem 7.4.3]{Gor} yields  
 that $C_G(Q)$ is $2$-nilpotent.

{\sf Assume Case (3):}
Set $Q:=\langle x \rangle $.
Since $x \,{\not\in}\,P_0$, we can write $x=a^ib^j t$ and for integers $i, j$. Then $C_G(x)=C_G(Q)$. 
Recall that  $C_P(x)= \langle x, ab \rangle$ by \cite[Lemma 3]{BW}.

\indent\indent
{\sf Assume first that $i+j$ is odd.} 
Then $x^2=(ab)^{i+j}$ is a generator for $Z(P)= \langle ab \rangle \cong C_{2^n}$. So that 
$|x|=2^{n+1}$. Moreover 
since $C_P(x)= \langle x, ab \rangle$, 
$C_P(x)= \langle x\rangle$. 
Hence $C_P(x)$ is a cyclic Sylow $2$-subgroup of $C_G(x)$, so by 
the Burnside normal $p$-complement theorem \cite[Theorem 7.6.1]{Gor}), 
$C_G(x)$ is $2$-nilpotent.

\indent\indent
{\sf Next, assume that $i+j$ is even.} 
Then first of all, by Lemma \ref{CentralizerOutsideP0}, 
$|C_P(x)|=2^{n+1}$. Also,
$x^2=(ab)^{i+j}$ and $\langle x^2 \rangle\lneqq Z(P)$. 
Moreover, since $C_P(x)=\langle x, ab \rangle$, any element 
of $C_P(x)$ is of the form $(ab)^k a^ib^j t = a^{i+k}b^{j+k} t$.
Since $i+j$ is even, we can write $i+j=2m$ for an integer $m$. So, set $k':=2^{n-1}-m$. 
Then, $2^n=2(m+k')=2m+2k'=i+j+2k'$.
Set $y := a^{i+k'}b^{j+k'} t$.
Then, first of all, $y\in C_P(x)$
since $y=(ab)^{k'}{\cdot}a^i b^j t\in\langle ab, x\rangle$.
Further, second of all, 
$y$ is an involution which is not in $Z(P)$
since $y^2 = a^{i+j+2k'} b^{i+j+2k'} = (ab)^{i+j+2k'}=(ab)^{2^n}=1$.
Hence by the order argument 
$C_P(x)= \langle y \rangle \times\langle ab \rangle =\langle y \rangle \times Z(P) \cong C_2 \times C_{2^n}$ 
(recall that $n \geq 2$).  
Since ${\mathrm{Aut}}(C_2 \times C_{2^n})$ is a $2$-group by \cite[Corollary 2.3 (4)]{CG},
and since $C_G(x)$ has the Sylow $2$-subgroup $C_P(x) \cong C_2\times C_{2^n}$, 
it holds that the normalizer and the centralizer of the Sylow $2$-subgroup of $C_G(x)$ are equal,
and hence again the Burnside normal $p$-complement theorem 
\cite[Theorem 7.4.3]{Gor} 
implies that $C_G(x)=C_G(Q)$ is $2$-nilpotent. 

{\sf Assume Case (4):}
Obviously, $C_P(Q)=Q \unlhd C_G(Q)$.
On the other hand, we have known that $C_P(Q)\in{\mathrm{Syl}}_2(C_G(Q))$.

Namely, $C_G(Q)$ has the Sylow $2$-subgroup $Q$ which is central.
Thus, the Schur-Zassenhaus theorem implies that $C_G(Q)$ is $2$-nilpotent. 

{\sf Assume Case (5):} Then $C_P(Q)=C_P(t)=Z(P)\times \langle t \rangle \cong C_{2^n} \times C_2$. 
Similar to the Case (2) above, we conclude that $C_G(Q)$ is $2$-nilpotent.
\end{proof}

\begin{Lemma}\label{2abelian}
Suppose that $G, P, Q$ satisfy the assumption/condition in
Lemma \ref{C_G(Q)2-nilp}(\ref{2abelianLemma}).
Further, assume that $G'$ is a finite group such that
$P\in{\mathrm{Syl}}_2(G')$ and that $\CF_P(G)=\CF_P(G')$.
Set $\mathfrak G:= G \times G'$. Then, 
$\Res\,_{\Delta Q\,C_{\mathfrak G}(\Delta Q)}^{N_{\mathfrak G}(\Delta Q)}
\Big( {\mathrm{Sc}}(N_{\mathfrak G}(\Delta Q), N_{\Delta P}(\Delta Q)) \Big)$ is indecomposable.
\end{Lemma}

\begin{proof}

First of all, note that $\CF_P(G)= \CF_P(G')\cong\CF_{\Delta P}(\mathfrak G)$
by the assumption,
so that the final one is saturated. 
Obviously $\Delta Q$ is fully 
$\CF_{\Delta P}(\mathfrak G)$-normalized in $\Delta P$ since $Q$ is fully $\CF_P(G)$-normalized.
Now, since $C_{\mathfrak G}(\Delta Q)=C_G(Q)\times C_{G'}(Q)$, it follows from
Lemma  \ref{C_G(Q)2-nilp} that $C_{\mathfrak G}(\Delta Q)$ is $2$-nilpotent.
Hence the assertion follows by Lemma \ref{2-nilpotent}.
\end{proof}

\begin{Lemma} \label{homocyclic}
Let $G$ and $G'$ be finite groups such that $P\in{\mathrm{Syl}}_2(G)\cap{\mathrm{Syl}}_2(G')$,
$P \cong C_{2^n}\wr C_2$ for some $n\geq 2$, and 
$\mathcal F:=\mathcal F_P(G)=\mathcal F_P(G')$ and $\mathfrak G:=G\times G'$.

Assume that $Q\leq P$  is fully $\CF$-normalized.
Recall that $P_0$ is the base subgroup of $P$,  and let
$M:={\mathrm{Sc}}(\mathfrak G, \Delta P)$.

Suppose furthermore that the following two things hold:
\begin{enumerate}
\item[\rm (a)]
$M(\Delta Q)={\mathrm{Sc}}(N_{\mathfrak G}(\Delta Q), N_{\Delta P}(\Delta Q))$,
\item[\rm (b)]
$M(\Delta P_0)={\mathrm{Sc}} (N_{\mathfrak G}(\Delta P_0), N_{\Delta P}(\Delta P_0))$.
\end{enumerate}

Assume moreover that one of the following holds:

\begin{enumerate}
\item[\rm\tt (I)] 
$Q\leq P_0=C_P(Q)$, and 
$N_\mathfrak G(\Delta Q)/C_\mathfrak G(\Delta Q)\cong N_{\Delta P}(\Delta Q)/C_{\Delta P}(\Delta Q)$.
\item[\rm\tt (II)]
$Q:=\langle a^{2^m} \rangle \times \langle b^{2^m} \rangle$ for some 
$m$ with 
$0 \leq m \leq n-1$ and $N_\mathfrak G(\Delta Q)/C_\mathfrak G(\Delta Q)\cong S_3$.
\end{enumerate}
Then, we have that 
$\Res\,_{\Delta Q \, C_\mathfrak G(\Delta Q)}^{N_\mathfrak G(\Delta Q)} (M(\Delta Q))$ is indecomposable.
\end{Lemma}

\begin{proof}
Recall that ${\mathcal F}_{\Delta P}(\mathfrak G) \cong {\mathcal F}$
and note that $C_\mathfrak G(\Delta Q)=C_G(Q) \times C_{G'}(Q)$.
Furthermore, 
$$O_{2'}(C_\mathfrak G(\Delta Q)) \text{ char }C_\mathfrak G(\Delta Q) 
\unlhd N_\mathfrak G(\Delta Q)$$
so that 
$O_{2'}(C_\mathfrak G(\Delta Q)) \unlhd N_\mathfrak G(\Delta Q)$. Hence,  
$O_{2'}(C_\mathfrak G(\Delta Q)) \leq O_{2'}(N_\mathfrak G(\Delta Q))$. Therefore, 
$O_{2'}(C_\mathfrak G(\Delta Q)) \leq 
\ker\Big({\mathrm{Sc}}(N_\mathfrak G(\Delta Q), N_{\Delta P}(\Delta Q))\Big)$
since the Scott module is in the principal $2$-block.
So, by using an easy inflation argument, it is enough to prove the lemma
when $O_{2'}(C_\mathfrak G(\Delta Q))=1$,

Namely, we can assume that 
\begin{equation}\label{2'coreFree}
O_{2'}(C_G(Q))=1 \text{ and }O_{2'}(C_{G'}(Q))=1.  
\end{equation}
Note that $C_P(Q)=P_0$ for both cases {\tt (I)} and {\tt (II)}.
Now, since $Q$ is fully $\CF$-normalized in $P$

it follows from Lemma \ref{Lin2.5} that 
$P_0 \in{\mathrm{Syl}}_2(C_G(Q))\cap{\mathrm{Syl}}_2(C_{G'}(Q))$. 
Hence we have 
 by (\ref{2'coreFree}) and Theorem 1 of \cite{B} that $P_0\unlhd C_G(Q)$ and $P_0\unlhd C_{G'}(Q)$,
so that 
\begin{equation}\label{P0xP0Normal}
P_0 \times P_0 \unlhd C_\mathfrak G(\Delta Q)=:C.
\end{equation}

On the other hand, since $P_0$ is fully $\CF$-normalized in $P$,
we know by Lemma \ref{2abelian} that
\begin{equation}\label{ScottDeltaP_0}
\Res\,_{C_{\mathfrak G}(\Delta P_0)}^{N_{\mathfrak G}(\Delta P_0)}
\Big( {\mathrm{Sc}}(N_{\mathfrak G}(\Delta P_0), N_{\Delta P}(\Delta P_0)) \Big) 
\text{ is indecomposable.}
\end{equation}
Recall that 
\begin{equation}\label{BrauerConstruction}
\text{''taking Brauer construction'' and ''taking restriction'' commute.}
\end{equation}

Set $M_1:=M(\Delta Q), N:=N_\mathfrak G(\Delta Q)$ and $C:=C_\mathfrak G(\Delta Q)$. 
Since $C_P(Q)=P_0$, 
$$ C_{\Delta P}(\Delta Q)= \Delta P_0.$$
Hence 
by Lemma \ref{centralizer}(ii),
\begin{equation}\label{NtoC}
{\mathrm{Sc}}(C, \Delta P_0)\,\Big|\,{\mathrm{Res}}\,^N_C
\,\Big({\mathrm{Sc}}(N, N_{\Delta P}(\Delta Q))\Big).
\end{equation} 
On the other hand, by the definition of $M_1$ and the assumption (a),
$$\Res\,^N_C(M_1)=\Res\,^N_C(M(\Delta Q)) = 
\Res\,^N_C\Big({\mathrm{Sc}}(N, N_{\Delta P}(\Delta Q))\Big),
$$
so that 
\begin{equation}\label{M1}
\Res\,^N_C(M_1)\,\Big|\,\Res\,^N_C\circ\Ind_{N_{\Delta P}(\Delta Q)}^N\,(k).
\end{equation}

Hence,
$ {\mathrm{Sc}}(C, \Delta P_0)\,|\, \Res\,^N_C\,(M_1)$ and we can write that
\begin{equation}\label{XX}
\Res\,_C^N(M_1) =\Sc(C, \Delta P_0) \bigoplus X \text{\quad for a }kC\text{-module }X.
\end{equation}

{\sf Case \tt (I):}
{\sf We claim here that} 
\begin{equation}\label{N}
N=C{\cdot}N_{\Delta P}(\Delta Q).
\end{equation}
By the assumption {\tt (I)}, $|N/C|=|N_P(Q)/C_P(Q)|=|N_P(Q)/P_0|$, so that $|N/C|=1$ or $2$.
If $N=C$, then the claim is trivial. So we may assume that $|N/C|=2$.
Assume that $N \gneqq C{\cdot}N_{\Delta P}(\Delta Q)$. 
Apparently, $|N/C|\geq |N_{\Delta P}(\Delta Q)\backslash N/C|$, so that 
$|N_{\Delta P}(\Delta Q)\backslash N/C|=2$. Thus, $N_{\Delta P}(\Delta Q){\cdot}C=C$.
Hence, $\Delta N_P(Q) = N_{\Delta P}(\Delta Q) \leq C=C_G(Q)\times C_{G'}(Q)$, so that
$N_P(Q)\leq C_P(Q)$, which implies that $N_P(Q)=C_P(Q)$, a contradiction since
we are assuming that $2=|N/C|=|N_{\Delta P}(\Delta Q)/C_{\Delta P}(\Delta Q)|=|N_P(Q)/C_P(Q)|$.
So, {\sf the claim is proved.}

By (\ref{M1}), the Mackey formula and (\ref{N}), 
$$\Res\,_C^N(M_1) \ \Big| \ (\Res\,_C^N\circ\Ind_{N_{\Delta P}(\Delta Q)}^N)(k)
= \Ind_{C\cap N_{\Delta P}(\Delta Q)}^C(k)=\Ind_{C_{\Delta P}(\Delta Q)}^C(k).$$ 
Hence, by noting that $C_P(Q)=P_0$ by the assumption {\tt (I)}, we have
\begin{equation}\label{ResM1}
\Res\,_C^N(M_1) \ \Big| \ \Ind_{\Delta P_0}^C(k).
\end{equation}

Thus, $X \, | \, \Res\,^N_C(M_1) \,| \, \Ind_{\Delta P_0}^C(k)$ by (\ref{XX}) and (\ref{ResM1}), 
so that $X \ | \  \Ind_{\Delta P_0}^C(k)$.  So, let us restrict these $kC$-modules 
to $P_0\times P_0$. Namely,
\begin{equation}\label{ResToP0xP0}
\Res\,_{P_0 \times P_0}^C(X) \ {\Big |} \ ( \Res\,_{P_0 \times P_0}^C\circ\Ind_{\Delta P_0}^C) (k).
\end{equation}
By the Mackey formula  
$$(\Res\,_{P_0 \times P_0}^C\circ\Ind_{\Delta P_0}^C)(k)
=\bigoplus_{c} \Ind_{(P_0 \times P_0) \, \cap\, {}^c (\Delta P_0)}^{P_0 \times P_0}(k)$$
where $c$ runs through the double cosets of $[(P_0 \times P_0) \backslash C / \Delta P_0]$. 
Note by (\ref{P0xP0Normal}) that 
$(P_0 \times P_0) \cap {}^c (\Delta P_0)
={}^c(P_0 \times P_0) \cap {}^c (\Delta P_0)
= {}^c((P_0 \times P_0) \cap \Delta P_0)={}^c(\Delta P_0)$.
Hence we can rewrite 
$$(\Res\,_{P_0 \times P_0}^C\circ\Ind_{\Delta P_0}^C)(k)
=\bigoplus_{c} \Ind_{{}^c(\Delta P_0)}^{P_0\times P_0}(k)$$

{\sf Now, we claim next that $X=0$. So assume that $X\,{\not=}\,0$.}
Then we get by (\ref{ResToP0xP0}) that 
$$\Res\,_{P_0 \times P_0}^C(X) \,\Big|\, \bigoplus_{c} \Ind_{{}^c(\Delta P_0)}^{P_0\times P_0}(k).$$
Recall that each direct summand $\Ind\, _{{}^{c}(\Delta P_0)}^{P_0 \times P_0}(k)$ is indecomposable with
the vertex $^{c}(\Delta P_0)$ by Green's indecomposability theorem. Hence by the Krulll-Schmidt theorem,
there are elements $c_1, \cdots, c_m \in C$ for some $m\geq 1$ such that
\begin{equation}\label{ResX}
\Res\,_{P_0 \times P_0}^C(X) 
= \bigoplus_{i=1}^m \Ind\, _{{}^{c_i}(\Delta P_0)}^{P_0 \times P_0}(k).
\end{equation}
Since we are assuming that $X\,{\not=}\,0$ and hence $m\geq 1$, set $c:=c_1$.
As noted above, $\Ind\, _{{}^{c}(\Delta P_0)}^{P_0 \times P_0}(k)$
is an indecomposable $k(P_0\times P_0)$-module that has the vertex 
${{}^{c}}(\Delta P_0)$ and has a trivial source. Hence, \cite[(27.7) Corollary]{Th} implies that
$\Big( \Ind\, _{{}^{c}(\Delta P_0)}^{P_0 \times P_0}(k)\Big)({{}^{c}}(\Delta P_0))\,{\not=}\,0$,
so that $X ({{}^{c}}(\Delta P_0))\,{\not=}\,0$. Now since $X$ is a $kC$-module and $c\in C$,
we have that
$ 0\,{\not=}\,X ({{}^{c}}(\Delta P_0)) = {{}^{c}}X({{}^{c}}(\Delta P_0))={{}^{c}}X( \Delta P_0)$,
so that 
\begin{equation}\label{DeltaP0}
X(\Delta P_0)\,{\not=}\,0.
\end{equation}

Next, let us apply the Brauer construction $-(\Delta P_0)$ to (\ref{XX}). Then
\begin{equation*}
(\Res\,_C^N(M_1) )(\Delta P_0)=(\Sc(C, \Delta P_0))(\Delta P_0) \bigoplus X(\Delta P_0).
\end{equation*} 
Since $\Delta P_0$ acts trivially on $\Sc(C, \Delta P_0)$,
$ (\Sc(C, \Delta P_0))(\Delta P_0)=\Sc(C, \Delta P_0)$. Hence we can rewrite the above as
\begin{equation}\label{BrauerConst1}
(\Res\,_C^N(M_1) )(\Delta P_0)=\Sc(C, \Delta P_0)\bigoplus X(\Delta P_0).
\end{equation} 
By (\ref{DeltaP0}), 
\begin{equation}\label{RHS}
\text{the right hand side of (\ref{BrauerConst1}) is not indecomposable.}
\end{equation}
Now, let us look at the left hand side of (\ref{BrauerConst1}).
But before it, let us remark that since $\Delta Q\unlhd\Delta P_0$, it follows from Lemma \ref{Kunugi} 
that
$$(M(\Delta Q))(\Delta P_0) \cong 
     \Res\,^{N_{\mathfrak G}(\Delta P_0)}_{N_{\mathfrak G}(\Delta P_0)\cap N_{\mathfrak G}(\Delta Q)}
     \Big( M(\Delta P_0) \Big).
$$ 
Recall the definition of $N:=N_{\mathfrak G}(\Delta Q)$, so that the above means that
\begin{equation}\label{BrouePuig}
(M(\Delta Q))(\Delta P_0) \cong \Res\,^{N_{\mathfrak G}(\Delta P_0)}_{N_N(\Delta P_0)}
\Big( M(\Delta P_0) \Big).
\end{equation}

\begin{align*}
(\Res\,_C^N(M_1) )(\Delta P_0) 
&\cong \Res\,^{N_N(\Delta P_0)}_{C_N(\Delta P_0)}\,( M_1(\Delta P_0))
\text{ by (\ref{BrauerConstruction})}
\\
&= \Res\,^{N_N(\Delta P_0)}_{C_N(\Delta P_0)}\,\Big( (M(\Delta Q))\, (\Delta P_0) \Big)
\\
&\text{ by the definition of }M_1
\\
&\cong \Res\,^{N_N(\Delta P_0)}_{C_N(\Delta P_0)}\,
              \Big(\Res\,^{N_{\mathfrak G}(\Delta P_0)}_{N_N(\Delta P_0)}\Big( M(\Delta P_0) \Big) \Big)
\  \text{ by (\ref{BrouePuig})}
\\
&= \Res\,^{N_{\mathfrak G}(\Delta P_0)}_{C_N(\Delta P_0)}\,\Big( M(\Delta P_0)\Big)
\\
&\cong 
\Res\,^{N_{\mathfrak G}(\Delta P_0)}_{C_N(\Delta P_0)}\,
\Big( {\mathrm{Sc}} (N_{\mathfrak G}(\Delta P_0), N_{\Delta P}(\Delta P_0)\Big)
\\
\text{ since }
M(\Delta P_0)&={\mathrm{Sc}} (N_{\mathfrak G}(\Delta P_0), N_{\Delta P}(\Delta P_0))
\text{ by the assumption (b)}.
\end{align*}
Now note that
$C_N(\Delta P_0)=C_{\mathfrak G}(\Delta P_0)\cap N 
=C_{\mathfrak G}(\Delta P_0)\cap N_{\mathfrak G}(\Delta Q)
=C_{\mathfrak G}(\Delta P_0)$
since $C_{\mathfrak G}(\Delta P_0)\leq C_{\mathfrak G}(\Delta Q) \leq N_{\mathfrak G}(\Delta Q)$
(recall that $P_0\geq Q) )$. Namely,
$$ C_N(\Delta P_0)=C_{\mathfrak G}(\Delta P_0).$$
So from the above
\begin{equation}\label{finalStroke}
(\Res\,_C^N(M_1) )(\Delta P_0)  \cong 
\Res\,^{N_{\mathfrak G}(\Delta P_0)}_{C_{\mathfrak G}(\Delta P_0)}\,
\Big( {\mathrm{Sc}} (N_{\mathfrak G}(\Delta P_0), N_{\Delta P}(\Delta P_0)\Big).
\end{equation}

Apparently by (\ref{finalStroke}) and (\ref{ScottDeltaP_0}), $(\Res\,_C^N(M_1) )(\Delta P_0)$
is indecomposable. But this is a contradiction by (\ref{RHS}).
{\sf Hence we have proved the claim that $X=0$.}

Therefore by (\ref{XX}), $\Res\,_C^N(M_1) =\Sc(C, \Delta P_0)$, so that by the definition of $M_1$,
we finally have that 
$$\Res\,_C^N(M(\Delta Q)) =\Sc(C, \Delta P_0).$$
Since $\Delta Q$ acts trivially on $M(\Delta Q)$, this is the equivalent to say that
$$\Res\,_{{\Delta Q}\cdot C}^N(M(\Delta Q)) =\Sc(C, \Delta P_0).$$
This completes the proof of {\sf the case \tt (I)}.

{\sf Case \tt (II):}
First, note that for any $x\in N$ since $C\unlhd N$ and $C_P(Q)=P_0$ by {\tt (II)},
$$ C\cap{^x\!(\Delta P)} = {^x\!C}\cap{^x\!(}\Delta P)={^x\!(}C\cap\Delta P)
={^x\!(}C_{\mathfrak G}(\Delta Q)\cap\Delta P) = {^x\!(}\Delta P_0).$$
By (\ref{XX}) and (\ref{M1}), since $N_P(Q)=P$, 
$$X\,\Big|\, \Res\,^N_C(M_1)\,\Big|\,\Res\,^N_C\circ\Ind_{\Delta P}^N\,(k).
$$
Hence by (\ref{XX}), the Mackey decomposition and the above, we have
$$
X\,\Big|\,\bigoplus_x\,\Ind\,_{{^x\!(}\Delta P_0)}^C (k)
$$ 
where $x$ runs through all representatives of the double cosets  
$[C\backslash N/\Delta P]$.

{\sf We claim here also that $X=0$, so assume that $X\,{\not=}\,0$.}
Then, there is an indecomposable $kC$-module $Y$ such that
$Y\,|\,X$, so that by the above and Krull-Remak's theorem, 
$$
Y\,\Big|\,\Ind\,_{{^x\!(}\Delta P_0)}^C (k)
\text{\qquad for some }x\in N.
$$ 
Then we restrict these to $P_0\times P_0$, that is,
\begin{equation}\label{ResY}
\Res\,_{P_0 \times P_0}^C(Y) \,\Big| \, 
( \Res\,_{P_0 \times P_0}^C\circ\,\Ind\,_{^x\!(\Delta P_0)}^C)(k).
\end{equation}
Then, by the Mackey formula, we have that   
\begin{equation}\label{C->P0xP0}
(\Res\,_{P_0 \times P_0}^C\circ\Ind\,_{^x\!(\Delta P_0)}^C)(k)
=\bigoplus_{c} \Ind\,_{(P_0 \times P_0) \, \cap\, {}^c (^x\!(\Delta P_0)) }^{P_0 \times P_0}(k)
\end{equation}
where $c$ runs through representatives of the double cosets of 
$[(P_0 \times P_0)\, \backslash \,C\, /\, {^x\!(}\Delta P_0)]$. 
Now, note that $P_0$ is an $\CF$-essential subgroup of $P$
from \cite[Theorem 5.3]{CG} since 
\linebreak
$N_\mathfrak G(\Delta Q)/C_\mathfrak G(\Delta Q)\cong S_3$ by
the assumption in {\tt (II)}.
So that Lemma \ref{essential} implies that 
$$N:=N_{\mathfrak G}(\Delta Q)\leq N_{\mathfrak G}(\Delta P_0){\cdot}C_{\mathfrak G}(\Delta Q)
 = N_{\mathfrak G}(\Delta P_0)\cdot C.
$$
Hence we can write $x = c' n'$ for some $c'\in C$ and $n'\in N_{\mathfrak G}(\Delta P_0)$.
Thus, for any $c\in C$,
$$
(P_0 \times P_0) \, \cap\, {}^{cx}\!(\Delta P_0) =(P_0 \times P_0) \, \cap\, {}^{cc'n'}\!(\Delta P_0) 
=(P_0 \times P_0) \, \cap\, {}^{cc'}\!(\Delta P_0).
$$
Thus, from (\ref{ResY}) and (\ref{C->P0xP0}),
$$
\Res\,_{P_0 \times P_0}^C(Y) \,\Big| \, 
\bigoplus_{y\in\mathcal Y} \Ind\,_{(P_0 \times P_0) \, \cap\, {}^y\!(\Delta P_0) }^{P_0 \times P_0}(k)
$$
where $\mathcal Y$ is a some non-empty subset of $C$.
Since for each $y$, $\Ind\,_{(P_0 \times P_0) \, \cap\, {}^y\!(\Delta P_0) }^{P_0 \times P_0}(k)$
is indecomposable by Green's theorem, we actually get by the Krull-Schmidt theorem that 
$$
\Res\,_{P_0 \times P_0}^C(Y) \,=\,
\bigoplus_{y\in\mathcal Y'} \Ind\,_{(P_0 \times P_0) \, \cap\, {}^y\!(\Delta P_0) }^{P_0 \times P_0}(k)
$$
where $\mathcal Y'$ is a some non-empty subset of $C$.
This is the same as in (\ref{ResX}) in the proof of {\tt (I)}.
Therefore, just doing the same things after (\ref{ResX}), we do have also that
$Y=0$, {\sf a contradiction.} Hence we are done also the proof of {\sf Case\tt (II)}.
\end{proof}

\section{The proof of the main theorem}

\noindent
Now we are ready to prove our main result.

\begin{proof}[{\bf Proof of Theorem \ref{product}}]
Set $\mathfrak G:=G \times G'$ and  $M:={\mathrm{Sc}}(\mathfrak G, \Delta P)$.
Since $P\in{\mathrm{Syl}}_2(G)$,  ${\mathcal F}:={\mathcal F}_P(G)$ is a saturated fusion system
(see Proposition 1.3 of \cite{BLO}). Furthermore, ${\mathcal F}_{\Delta P}(\mathfrak G) $ is 
also saturated because  ${\mathcal F}_{\Delta P}(\mathfrak G) \cong {\mathcal F}_P(G)$
since $\CF_P(G)=\CF_P(G')$.

By \cite[Lines 12 $\sim$ 18 of the proof of
Theorem 1.3 on p.445]{IK}, it 
is enough to prove our claim for fully ${\mathcal F}_{\Delta P}(\mathfrak G) $--normalized subgroups of $\Delta P$.
Let $\Delta Q \leq \Delta P$ be any 
fully ${\mathcal F}_{\Delta P}(\mathfrak G) $-normalized 
subgroup of $\Delta P$. We shall prove that 
$\Res\,^{N_{\mathfrak G}(\Delta Q)}_{\Delta Q\,C_{\mathfrak G}(\Delta Q)} (M(\Delta Q))$ is 
indecomposable for any fully $\CF$-normalized subgroup of $P$
by using induction on $|P:Q|$.

If $|P:Q|=1$, the assertion holds by \cite[Lemma 4.3(b)]{KKM}.

Now, assume that $Q \lneqq P$ and that $M(\Delta R)$ is indecomposable 
as a $k(\Delta R \cdot C_{\mathfrak G}(\Delta R))$-module for all fully $\CF$-normalized subgroups $R$ with 
$|P:R|<|P:Q|$.

{\sf We first claim that} $M(\Delta Q)$ is indecomposable as a $k\,N_{\mathfrak G}(\Delta Q)$-module. 
Suppose that 
$$M(\Delta Q)=M_1 \bigoplus \ldots \bigoplus M_m$$ 
for some $m\geq 1$ such that each $M_i$ is an indecomposable 
$k\,N_{\mathfrak G}(\Delta Q)$-module. 
By Lemma 3.1 of \cite{IK} and Theorem 1.7 of \cite{K}, we have
that ${\mathrm{Sc}}(N_{\mathfrak G}(\Delta Q), N_{\Delta P}(\Delta Q))\,\Big|\,M(\Delta Q)$. So we can set 
$$M_1:= {\mathrm{Sc}}(N_{\mathfrak G}(\Delta Q), N_{\Delta P}(\Delta Q)).$$
Further,
$M(\Delta Q) \ | \ {\mathrm{Res}}\,_{N_{\mathfrak G}(\Delta Q)}^{\mathfrak G} (M) $
(see \cite[Lemma 2.1(ii)]{IK}).

{\sf Now assume that $m\geq 2$.}
Hence we have that
$ M_2\ | \ {\mathrm{Res}}\,_{N_{\mathfrak G}(\Delta Q)}^{\mathfrak G} (M)$. 
By Burry-Carlson-Puig's theorem (see \cite[Theorem 4.4.6]{NT}),
a vertex of $M_2$ is not equal to $\Delta Q$ since $Q \lneqq P$. 
Since $M_2 \, | \, M(\Delta Q)$, \cite[Lemma 2.1(i)]{IK} implies that 
a vertex of $M_2$ contains $\Delta Q$. As a result, a vertex of $M_2$, 
say $\Delta V$, satisfies $\Delta V \gneqq \Delta Q$. Moreover,
since $M \ {|} \ \mathrm{Ind}_{\Delta P}^{\mathfrak G}(k)$, 
we have by the Mackey decomposition
$$
M_2 \ {\Big |} \ \bigoplus_h {\mathrm{Ind}}
_{ N_{\mathfrak G}(\Delta Q) \cap\,^h\!(\Delta P)}^ {N_{\mathfrak G}(\Delta Q)} (k)
$$
where $h$ runs over representatives
of the double cosets in $[N_{\mathfrak G}(\Delta Q) \backslash {\mathfrak G} / \Delta P]$
which satisfies $\Delta Q \leq \ ^h\!(\Delta P)$ by 1.4 of \cite{Bro}. 
Hence $\Delta V$
lies in $N_{^h\!(\Delta P)}(\Delta Q)=N_{\mathfrak G}(\Delta Q)\, \cap{^h\!(\Delta P)}$ for some $h\in\mathfrak G$.
Note that for such $h$, we have that
$$
 \Delta V \leq N_{^h\!(\Delta P)}(\Delta Q) 
\leq_{N_{\mathfrak G}(\Delta Q)}  N_{\Delta P}(\Delta Q)$$
where the second inequality follows from Lemma 3.2 of \cite{IK} (since $\Delta Q$ is fully $\CF$-normalized implies 
that $\Delta Q$ is fully automized and receptive). 
Hence $\Delta V \leq N_{\mathfrak G}(\Delta Q)$, which implies that 
$\Delta Q$ is a proper normal subgroup of $\Delta V$.
Then, since $N_{\Delta P}(\Delta Q)$ is a vertex of $M_1$ by the definition of $M_1$,
we have by \cite[(27.7)Corollary]{Th} that
$$M_1(\Delta V) \neq 0.$$
Furthermore, by \cite[(27.7)Corollary]{Th} we also get that 
$$M_2(\Delta V) \neq 0$$ 
since $\Delta V$ is a vertex of $M_2$.
Further by Lemma \ref{Kunugi}, 
$$ M_1(\Delta V) \bigoplus M_2(\Delta V) \ {\Big |} \  (M(\Delta Q))(\Delta V) 
\cong 
\ \Res\, ^{ N_{\mathfrak G}(\Delta V)}_{  N_{\mathfrak G}(\Delta V)\cap N_{\mathfrak G}(\Delta Q) } M(\Delta V).
$$
Since $V \geq Q$,  
$C_{\mathfrak G}(\Delta V)\leq C_{\mathfrak G}(\Delta Q)\leq N_{\mathfrak G}(\Delta Q)$,
and hence
$C_{\mathfrak G}(\Delta V) \leq N_{\mathfrak G}(\Delta V) \cap N_{\mathfrak G}(\Delta Q)$.
So the isomorphism above restricts to as  $k\,C_{\mathfrak G}(\Delta V) $-modules. 
This means that $M(\Delta V)$ has at least two non-zero direct summands as a $k\,C_{\mathfrak G}(\Delta V)$-module,
{\sf which contradicts} our 
induction hypothesis (note that since $Q \lneqq V$, we have that $|P:V| < |P:Q|$
and that $\Delta V$ can be changed to an $\CF$-conjugate of itself and can be made fully 
$\CF$-normalized in $\Delta P$).

Therefore, $m=1$ and $M(\Delta Q)$ is indecomposable
as a $k\,N_{\mathfrak G}(\Delta Q)$-module or equivalently 
\begin{equation}\label{M(DeltaQ)}
M(\Delta Q)\cong \Sc(N_{\mathfrak G}(\Delta Q), N_{\Delta P}(\Delta Q))
\text{ for all fully $\CF$-normalized subgroups }Q\leq P.
\end{equation}
Apparently, (\ref{M(DeltaQ)}) holds for the case $Q=P_0$.

Now, {\sf our final aim} is to prove that
\begin{equation}\label{aim}
\Res\,^{N_{\mathfrak G}(\Delta Q)}_{\Delta Q\,C_{\mathfrak G}(\Delta Q)} (M(\Delta Q))
\text{ is indecomposable}
\end{equation}
for any fully $\CF$-normalized
subgroup $Q$ with $1\,{\not=}\,Q\lneqq P$.

\item[$\bullet$] Assume first that $Q\lneqq P$ is non-abelian.

\noindent
Then, both $C_G(Q)$ and $C_{G'}(Q)$ 
are $2$-nilpotent  by Lemma \ref{C_G(Q)2-nilp}(\ref{nonabelian}).
So $C_{\mathfrak G}(\Delta Q)$ is $2$-nilpotent. Hence,
(\ref{aim}) holds
by Lemma \ref{2-nilpotent} and (\ref{M(DeltaQ)}).

\noindent
\item[$\bullet$] Next we look at the case that $Q$ is cyclic.

{\sf Case C1:} $Q \leq Z(P)$.
By (\ref{M(DeltaQ)}) and \cite[Lemma 4.4]{KKM}, we have (\ref{aim}). 

{\sf Case C2:} $Q\leq P_0$ and $Q\,{\not\leq}\,Z(P)${\sf .}
Obviously, $Q\leq P_0=C_P(Q)$.
Further note that since ${\mathrm{Aut}}(\Delta Q)$ is a $2$-group, 
Lemma \ref{Lin2.5} yields that
$$
N_{\Delta P}(\Delta Q)/C_{\Delta P}(\Delta Q)
\cong N_{\mathfrak G}(\Delta Q)/C_{\mathfrak G}(\Delta Q).
$$
Thus,
(\ref{aim}) follows from (\ref{M(DeltaQ)}) and
Lemma \ref{homocyclic}{\tt (I)}.

{\sf Case C3: $Q\,{\not\leq}\, P_0$.}
By Lemma \ref{C_G(Q)2-nilp}(\ref{cyclic}), 
both $C_G(Q)$ and $C_{G'}(Q)$ are $2$-nilpotent, so that so is 
$C_{\mathfrak G}(\Delta Q)$. Hence, Lemma \ref{2-nilpotent} implies (\ref{aim}).

\item[$\bullet$]
Thus, from now on {\sf we assume that $Q$ is non-cyclic abelian.}
Note that then $2$-rank of $Q$ is precisely $2$.

\item[$\bullet$] Next we consider the case that {\sf $Q$ is non-homocyclic}
(so that our situation says that  $Q\cong C_{2^m}\times C_{2^{m'}}$ for $1\leq m <m'\leq n$).
Since we are assuming that $Q$ is abelian, Lemma \ref{list} implies that
$N_{\mathfrak G}(\Delta Q)/C_{\mathfrak G}(\Delta Q)$ is a $2$-group.
Further by Lemma \ref{Lin2.5},
$$ N_{\Delta P}(\Delta Q)/C_{\Delta P}(\Delta Q)\in
{\mathrm{Syl}}_2(N_{\mathfrak G}(\Delta Q)/C_{\mathfrak G}(\Delta Q)).$$
Then, since ${\mathrm{Aut}}(Q)$ is a $2$-group (see \cite[Lemma 1]{Sam12}),
we get even that 
\begin{equation}\label{nonHomocyclic}
N_{\Delta P}(\Delta Q)/C_{\Delta P}(\Delta Q)=N_{\mathfrak G}(\Delta Q)/C_{\mathfrak G}(\Delta Q).
\end{equation}

{\sf Case NH1:} Assume first that $Q\,{\not\leq}\,P_0$. Then it follows from
Lemma \ref{C_G(Q)2-nilp}(\ref{non-homocyclic}) that $C_G(Q)$ and $C_{G'}$ are both
$2$-nilpotent, and hence so is $C_{\mathfrak G}(\Delta Q)$. Hence
Lemma \ref{2-nilpotent} yields (\ref{aim}).

{\sf Case NH2:} Assume next that $Q\leq P_0$.
Since $Q$ is not cyclic, $Q\,{\not\leq}\,Z(P)$ and hence $C_P(Q) \lneqq P$. 
Obviously, $P_0\leq C_P(Q)$, so that 
$$C_P(Q)=P_0.$$
Then for $N_P(Q)$ there are two possibilities, namely $N_P(Q)=P_0$ or $P$.

\indent\indent
{\sf Subcase NH2(a):} Suppose that $N_P(Q)=P_0$.
Then $N_P(Q)=C_P(Q)$, so that 
$N_{\Delta P}(\Delta Q)=C_{\Delta P}(\Delta Q)$,
which implies from (\ref{nonHomocyclic}) that $N_{\mathfrak G}(\Delta Q)=C_{\mathfrak G}(\Delta Q)$.
Hence (\ref{M(DeltaQ)}) automatically yields (\ref{aim}).

\indent\indent
{\sf Subcase NH2(b):} Suppose that $N_P(Q)=P$.
Then, it follows from (\ref{M(DeltaQ)}), (\ref{nonHomocyclic}) and Lemma \ref{homocyclic}{\tt (I)}
that (\ref{aim}) holds.

\item[$\bullet$] Finally we consider the case that {\sf $Q$ is homocyclic}
(so that our situation says that  $Q\cong C_{2^m}\times C_{2^{m}}$ for $1\leq m\leq n$).
In this case, first of all,  from Lemma \ref{Lin2.5}
\begin{equation}\label{HomocyclicSylow2}
N_{\Delta P}(\Delta Q)/C_{\Delta P}(\Delta Q)
\in{\mathrm{Syl}}_2(N_{\mathfrak G}(\Delta Q)/C_{\mathfrak G}(\Delta Q)).
\end{equation}

{\sf Case H1:} $Q\,{\not\leq}\,P_0$. 
Then { by Lemma \ref{HomocyclicSubgroup}(iii)}
$$Q=\langle (ab)^{2^{n-1}} \rangle \times \langle t \rangle \cong C_2 \times C_2.$$
Then by Lemma \ref{C_G(Q)2-nilp}(\ref{hcyclic}),
both $C_G(Q)$ and $C_{G'}(Q)$ are $2$-nilpotent, and hence so is $C_{\mathfrak G}(\Delta Q)$. Therefore, 
from Lemma \ref{2-nilpotent}, (\ref{aim}) holds.

{\sf Case H2:} $Q \leq P_0$. 
Then {from Lemma \ref{HomocyclicSubgroup}(i)--(ii)},  
$$Q=\langle a^{2^m} \rangle \times \langle b^{2^m} \rangle
\text{ for some }m \text{ with }0 \leq m \leq n-1.$$  
Then, $C_P(Q)=P_0$ and $N_P(Q)=P$, so $N_P(Q)/C_P(Q) \cong C_2$. 
Hence by (\ref{HomocyclicSylow2}),
\begin{equation}\label{2or6}
C_2 \cong N_{\Delta P}(\Delta Q)/C_{\Delta P}(\Delta Q)
\in{\mathrm{Syl}}_2(N_{\mathfrak G}(\Delta Q)/ C_{\mathfrak G}(\Delta Q)).
\end{equation}
Recall that $|{\mathrm{Aut}}(\Delta Q)|=2^\ell{\cdot}3$ for some integer $\ell\geq 0$
(see \cite[Line 1 p.5957]{CG}). Thus
$$ |N_{\mathfrak G}(\Delta Q)/ C_{\mathfrak G}(\Delta Q)|=2 \text{ or }6. $$
Namely, there are two possibilities.

\indent\indent
{\sf Subcase H2(a)}: Assume $|N_{\mathfrak G}(\Delta Q)/ C_{\mathfrak G}(\Delta Q)|=2$.
Then, from (\ref{2or6}),
$$N_{\mathfrak G}(\Delta Q)/ C_{\mathfrak G}(\Delta Q)=N_{\Delta P}(\Delta Q)/C_{\Delta P}(\Delta Q)
  \cong C_2.$$
Since $C_P(Q)=P_0\geq Q$, we know from 
(\ref{M(DeltaQ)}) and Lemma \ref{homocyclic}{\tt (I)} that (\ref{aim}) holds. 

\indent\indent
{\sf Subcase H2(b)}: Assume next that  $|N_{\mathfrak G}(\Delta Q)/ C_{\mathfrak G}(\Delta Q)|=6$.
Then, from (\ref{2or6}),
$$C_2\cong N_{\Delta P}(\Delta Q)/ C_{\Delta P}(\Delta Q)
\lneqq N_{\mathfrak G}(\Delta Q)/C_{\mathfrak G}(\Delta Q)\cong S_3.$$
Hence it follows from (\ref{M(DeltaQ)}) and  Lemma \ref{homocyclic}{\tt (II)} that
(\ref{aim}) holds.

\end{proof}

\end{document}